\documentclass[onefignum,onetabnum]{siamart190516}

\usepackage{latexsym}
\usepackage{graphicx}
\usepackage{mathptmx}
\usepackage{amsmath}
\usepackage{amsfonts}
\usepackage{amssymb}
\usepackage{amsbsy}
\usepackage{hyperref}

\usepackage{algorithmic}
\ifpdf
  \DeclareGraphicsExtensions{.eps,.pdf,.png,.jpg}
\else
  \DeclareGraphicsExtensions{.eps}
\fi

\newsiamremark{assumption}{Assumption}
\newsiamremark{remark}{Remark}
\newcommand{\iid}{\stackrel{\mathrm{i.i.d}}{\sim}}

\headers{BAYESIAN STOCHASTIC GRADIENT DESCENT}{T. Liu, Y. Lin, and E. Zhou}

\title{BAYESIAN STOCHASTIC GRADIENT DESCENT FOR STOCHASTIC OPTIMIZATION WITH STREAMING INPUT DATA
\thanks{A preliminary version of this paper appeared in Proceedings of the 2021 Winter Simulation Conference, 2021.
\funding{This research is funded by the Air Force Office of Scientific Research under Grant FA9550-19-1-0283, Grant FA9550-22-1-0244, and National Science Foundation under Grant DMS2053489.}
}
}

\author{Tianyi Liu \footnotemark[3]
\thanks{H. Milton School of Industrial and Systems Engineering, Georgia Institute of Technology, Atlanta, GA
30332 (\email{tliu341@gatech.edu}, \email{ylin429@gatech.edu}, \email{enlu.zhou@isye.gatech.edu}).}
\and Yifan Lin \footnote[3]{Equal contribution.} \footnotemark[2]
\and Enlu Zhou \footnotemark[2]}

\usepackage{amsopn}

\makeatletter
\newcommand*{\addFileDependency}[1]{% argument=file name and extension
  \typeout{(#1)}% latexmk will find this if $recorder=0 (however, in that case, it will ignore #1 if it is a .aux or .pdf file etc and it exists! if it doesn't exist, it will appear in the list of dependents regardless)
  \@addtofilelist{#1}% if you want it to appear in \listfiles, not really necessary and latexmk doesn't use this
  \IfFileExists{#1}{}{\typeout{No file #1.}}% latexmk will find this message if #1 doesn't exist (yet)
}
\makeatother

\ifpdf
\hypersetup{
  pdftitle={BAYESIAN STOCHASTIC GRADIENT DESCENT FOR STOCHASTIC OPTIMIZATION WITH STREAMING INPUT DATA},
  pdfauthor={Tianyi Liu, Yifan Lin, and Enlu Zhou}
}
\fi
\begin{document}

\maketitle

% REQUIRED
\begin{abstract}
We consider stochastic optimization under distributional uncertainty, where the unknown distributional parameter is estimated from streaming data that arrive sequentially over time. Moreover, data may depend on the decision of the time when they are generated. For both decision-independent and decision-dependent uncertainties, we propose an approach to jointly estimate the distributional parameter via Bayesian posterior distribution and update the decision by applying stochastic gradient descent on the Bayesian average of the objective function. Our approach converges asymptotically over time and achieves the convergence rates of classical SGD in the decision-independent case. We demonstrate the empirical performance of our approach on both synthetic test problems and a classical newsvendor problem.
\end{abstract}

% REQUIRED
\begin{keywords}
Bayesian estimation, streaming input data, stochastic gradient descent, endogenous uncertainty
\end{keywords}

% REQUIRED
\begin{AMS}
90C15
\end{AMS}

\section{Introduction}
\label{sec:intro}
Stochastic optimization is a mathematical framework that models decision making under uncertainty. It usually assumes that the decision maker has full knowledge about the underlying uncertainty through a known probability distribution and minimizes (or maximizes) a functional of the cost (or reward) function \cite{shapiro2021lectures}. However, the probability distribution of the randomness in the system is rarely known in practice and is often estimated from historic data. The impact of the estimation accuracy and the subsequent distributional uncertainty have been widely studied in the literature. For example, \cite{bonnans2013perturbation} and \cite{rachev2002quantitative} conduct perturbation analysis of the stochastic optimization problems and quantify the sensitivity of the optimal value (and/or solution) to the probability distribution. One popular approach to addressing this distributional uncertainty in stochastic optimization is distributionally robust optimization (DRO) (e.g. \cite{delage2010distributionally, bertsimas2018data, wiesemann2014distributionally}). The DRO framework assumes that the underlying unknown probability distribution  lies in an ambiguity set of probability distributions and then optimizes the problem with respect to the worst case in the ambiguity set. It has been successfully applied to a broad range of problems in statistics, optimization, and control, such as stochastic programming (e.g. \cite{bayraksan2015data, jiang2016data}), Markov Decision Processes (MDPs) (e.g. \cite{xu2010distributionally, yang2017convex}), stochastic control (e.g. \cite{van2015distributionally, yang2020wasserstein}), and ranking and selection (e.g. \cite{Gao2017, Xiao2018, Xiao2020, Fan2020}). To construct an appropriate ambiguity set that contains the true distribution with a probabilistic guarantee and ensures tractability of the optimization problem, various DRO methods have been developed, such as methods based on moment constraints (e.g., \cite{delage2010distributionally}), $\phi$-divergence (e.g. \cite{ben1987penalty}), and Wasserstein distance (e.g., \cite{esfahani2018data}). In contrast to DRO, \cite{Zhou2015, Wu2018} proposed a Bayesian risk optimization (BRO) framework, with the motivation to use the Bayesian posterior distribution (which encodes the likelihoods of all possibilities) to replace the ambiguity set (which treats every possibility inside the set with equal probability), and further take a risk functional with respect to the posterior distribution to allow more flexible risk attitude.

Nearly all the aforementioned works that focus on stochastic optimization in static setting assume that the input data are given as one fixed batch. However, in many applications, data are often collected over time, and the decision maker often needs to make decisions in an online fashion given all the available data. For example, an inventory manager observes the customer demand in a daily or weekly basis, and adjusts his/her decision accordingly; a robot that searches for an unknown source receives signals from the source over time, and makes its move accordingly (e.g. \cite{Li2022Bayesian}). Such streaming data have only been considered recently in stochastic simulation optimization, e.g.,  \cite{Wu2017}, \cite{Zhou2018}, \cite{Wu2019a}, \cite{Song2019a}. While these recent works consider the streaming input data, their assumption is that the data are generated from an exogenous (decision-independent) distribution and hence are independent and identically distributed (i.i.d.). This assumption restricts their application to many real-world problems where the input data are endogenous (decision-dependent). For example, in live streaming e-commerce, there is usually a rolling banner that counts how many products are left, and customers are more likely to purchase the product that has only a few left since it is more popular. As another example, in the supermarket, tall stacks of a product impact its visibility, which leads more customers to purchase the product \cite{giri1996inventory, balakrishnan2004stack}.

Motivated by these real-world problems where data arrive sequentially and could even depend on the decision, in this work we consider stochastic optimization problems where the underlying distribution is unknown but data from the distribution arrive in batches over time. We assume a parameterized distributional model, and thus the distribution family is known but the true distributional parameter is unknown. It is also interesting to consider a non-parametric setting with a prior of Dirichlet process (see \cite{wang2020nonparametric} for a non-parametric simulation optimization problem setting), though the associated analysis could be much more complicated. At each time stage, our procedure consists of two steps: 1) use the current batch of data to update the Bayesian posterior distribution of the distributional parameter, and 2) take the Bayesian average of the objective function and apply stochastic gradient descent (SGD) on this reformulated objective function. Our proposed approach can be viewed as an online extension of the BRO framework in \cite{Wu2018}: BRO considers a fixed batch of data and only need to solve the fixed BRO formulation; in contrast, we consider the setting where batches of data come in sequentially, and therefore, we update the stage-wise BRO problem every time with the new incoming data; moreover, due to the limited time in each stage, we can only apply a few SGD iterations to solve each stage-wise BRO problem. As a result, the convergence analyses of BRO and our paper are quite different and the results have distinct implications: the convergence of BRO shows that if the fixed batch of data has an infinite size, the BRO formulation recovers the true problem  and BRO solutions are indeed the true optimal solutions;  our convergence analysis shows that even though our algorithm applies SGD iterations to a sequence of estimated (Bayesian-average) problems, but the algorithm still converges to the true (local) optimal solution. Another related work \cite{shapiro2021bayesian} considers the same problem of fixed data batch as \cite{Zhou2015, Wu2018} and uses Bayesian average to estimate the true problem, but it also takes a robust approach with respect to the uncertainty associated with the parametric distributional model.

We consider both cases of exogenous and endogenous input data. In the former case, data follow a fixed distribution that only involves the distributional parameter. In the latter case, the data follow a time-varying distribution depending not only on the distributional parameter but also on the decision at the current time. It is worth noting that due to the correlation and non-stationarity of the decision-dependent data across time stages, the Bayesian estimation with such data is different from the classical Bayesian updating with i.i.d. data, which poses a great challenge to showing the consistency of the Bayesian posterior distribution. We consider the same problem as \cite{Song2019a}, but differ in two key aspects: first, we take a Bayesian approach to estimate the distributional parameter, whereas they estimate by maximum likelihood estimator (MLE) and solve the problem with the plug-in MLE; second, they only consider exogenous (decision-independent) uncertainty. Also note that compared to our preliminary conference version \cite{Liu2021Bayesian}, this paper is a substantial extension in both theoretical analysis and numerical experiments. For the decision-independent uncertainty, we further show the convergence rate of the proposed algorithm. Apart from a synthetic test problem, we also evaluate the performance of the proposed algorithm in a classical newsvendor problem.

Our considered problem is related to online learning (e.g. \cite{bottou1998online, shalev2011online}). Online learning is often formulated as a repeated game: at each round, the learner makes a prediction and receives the true solution (or a cost function), with the goal to minimize the cumulative cost over time. Classical algorithms in online learning such as Follow the Leader (FTL) and its variants, such as Follow the perturbed Leader (FTPL) and Follow the Regularized Leader (FTRL), incorporate the learning process, which takes the information from previous rounds to improve prediction, into the algorithms in order to choose the next action that leads to the lowest cumulative cost. In contrast to the goal of minimizing the cumulative cost, our considered problem aims to find an optimal solution of a stationary objective function in the decision-independent case and a non-stationary objective function in the decision-dependent case, where the non-stationarity is only caused by the decision-dependent uncertainty. Since the online data in our problem is restricted to the randomness in the system that is generated from the (unknown) underlying distribution, it is natural to update our belief of the (unknown) distribution in a Bayesian way. In addition to the distinctive goal in our problem, it is worth noting the key differences between our approach and two closely-related algorithms in online learning. The first one is the online gradient descent algorithm (see \cite{zinkevich2003online, bartlett2008adaptive, duchi2011adaptive}), for which the cost function can vary completely arbitrarily over time, and hence is unlike our SGD algorithm that makes use of the structure of the Bayesian average of the objective function over time. The second one is the Thompson sampling algorithm (see \cite{agrawal2012analysis, Bubeck2015Bandit}), which also assumes a parameterized model and updates the posterior distribution on the parameter in a Bayesian way. However, Thompson sampling makes the decision based on only one sample from the posterior distribution in each round; whereas our algorithms takes the entire posterior distribution into account and solves the Bayesian average of the original (unknown) objective function. Later in the numerical experiments, we show that the Bayesian average provides a better estimate of the original objective function compared to a point estimate. 

As a final note, the endogenous uncertainty has been considered in many fields, including dynamic programming (e.g. \cite{webster2012approximate}), robust optimization (e.g. \cite{nohadani2018optimization, Lappas2018}), and stochastic optimization (e.g. \cite{goel2006class, dupacova2006optimization, ekin2017augmented, hellemo2018decision, noyan2018distributionally, Luo2020, yu2020multistage}), with many applications in inventory control (e.g. \cite{benkherouf2001stochastic, lee2012newsvendor}), healthcare (e.g. \cite{green2013nursevendor}), and so on. However, almost none of the aforementioned work involving decision-dependent uncertainty take into consideration the additional input data. Only until recently, \cite{Juan2020} and \cite{dunner2020stochastic} study the performative prediction problem, which is essentially a stochastic optimization problem with streaming decision-dependent data; however, the goal is to find the so-called performatively stable point (or equilibrium point), which is in general different from the true optimal solution. Along the same line, \cite{Dmitriy2020} also considers static stochastic optimization under decision-dependent uncertainty, and proposes a proximal gradient method and its variants that converge to the performatively stable point under relatively strong assumptions (strong convexity, Lipschitz continuity, etc.). Asymptotic normality and optimality of the stochastic approximation algorithm are further studied in a follow-up work \cite{cutler2022stochastic}. Most recently, \cite{izzo2021learn} and \cite{miller2021outside} redesign the gradient algorithms in \cite{Juan2020} by introducing a gradient correction term, and show the convergence to the true optimal solution. In particular, \cite{izzo2021learn} also considers a parameterized model where the distributional parameter (as a function of the decision variable) can be estimated from streaming input data, and uses finite difference to estimate the gradient of the objective function. An important assumption in their approach is that the estimated distributional parameter has a constant error bound. Different from their approach, we learn the distributional parameter with a Bayesian approach, and show the Bayesian consistency of the posterior distribution that finally leads to the convergence of the SGD algorithm to a stationary point of the original objective function (optimal solution if the problem is convex).

We summarize the contribution of this paper as follows. First, we propose a Bayesian stochastic gradient descent approach to stochastic optimization problem with unknown underlying distribution and with streaming input data that could depend on the decision. This new approach is among the very few works \cite{Wu2018, shapiro2021bayesian, gupta2019near} in the literature that take a Bayesian perspective on approaching distributional uncertainty in stochastic optimization. Second, we show the convergence of our approach in the decision-independent case and decision-dependent case respectively. Under decision-independent uncertainty, our approach achieves the convergence rates of classical non-convex SGD. Third, we show the consistency of the Bayesian posterior distribution with endogenous non-i.i.d. data under mild conditions; this result is applicable to a wide range of  problems involving Bayesian estimation beyond the scope of this paper. Our non-asymptotic analysis of the Bayesian estimate with i.i.d. data is also new and could be potentially useful for analyzing other Bayesian algorithms. 

The rest of the paper is organized as follows. We first propose Bayesian-SGD algorithms for stochastic optimization with decision-independent and decision-dependent streaming input data in \cref{sec: algorithms}. We then analyze the convergence properties of the proposed algorithms for both cases in \cref{sec: convergence}. We verify the theoretical results and demonstrate the performance of our algorithms in the numerical experiments in \cref{sec: experiment}. Finally, we conclude the paper in \cref{sec: conclusions}.

\section{Bayesian SGD algorithms for stochastic optimization with streaming input data}
\label{sec: algorithms}
We consider the following stochastic optimization problems with decision-independent uncertainty and decision-dependent uncertainty, receptively:
\begin{align}\label{eq: obj-independent}
    \min _{x \in \mathcal{X}} H(x):=\mathbb{E}_{f(\cdot;\theta^c)}[h(x, \xi)] ~~~\text{(decision-independent uncertainty)}
\end{align}
\begin{align}\label{eq: obj-dependent}
    \min _{x \in \mathcal{X}} H(x):=\mathbb{E}_{f(\cdot;x,\theta^c)}[h(x, \xi)]~~~~~\text{(decision-dependent uncertainty)}
\end{align}
where $x \in \mathcal{X} \subset \mathbb{R}^d$ is the decision vector, $\xi \in \Xi \subset \mathbb{R}^m$ is a random vector, $h: \mathbb{R}^d \times \mathbb{R}^m \to \mathbb{R}$ is a deterministic function. The expectation is taken with respect to (w.r.t.) the distribution of $\xi$, which is denoted as $f(\cdot;\theta^c)$ in the decision-independent case, and as $f(\cdot;x,\theta^c)$ in the decision-dependent case. The density function $f(\cdot;x,\theta^c)$ takes a general form, where the parameter $\theta^c$ does not depend on $x$. For example, $f(\xi;x,\theta^c)=\theta^c x \exp (-\theta^c x \xi)$ is the density function of the exponential distribution with rate $\theta^c x$. More assumptions on the density function will be discussed in \cref{sec: convergence}. We assume the distribution of $\xi$ belongs to a parameterized family of distributions with parameter set $\Theta\subset \mathbb{R}^l$, and let $\theta^c$ be the true parameter value of the distribution.

In practice, the true distribution $f(\cdot;\theta^c)$, or in other words the true distributional parameter $\theta^c$, is rarely known exactly and usually estimated from data. We consider an online setting where data arrive sequentially in time and decisions are updated at each time stage. It is natural to take a Bayesian approach for sequential estimation of the unknown parameter, since it is computationally convenient and the estimate is guaranteed with strong consistency with i.i.d. data (however, Bayesian consistency with non-i.i.d. data are much more complicated, which we will discuss later in \cref{sec: convergence}). With the Bayesian estimate of the distributional parameter, we apply iterations of the SGD algorithm on the estimated problem to update the decision, because the light computational effort of SGD makes it appealing for the online setting. On a high level, at each time stage $t$, after observing a new batch of data we carry out the following two steps:

\begin{itemize}
\item Update the Bayesian posterior distribution of the parameter with the new data. 
\item Use SGD on the Bayesian average of problem \eqref{eq: obj-independent} or \eqref{eq: obj-dependent} to update the decision. 
\end{itemize}

We now discuss the details of these two steps in the following. Let's first focus on the decision-independent case. Suppose at each time stage $t$ we observe a batch of data $\mathbf{y}_t = \{y_{t,j}, j=1,\ldots,D\}$, where $\{y_{t,j}\}$ are i.i.d. according to $f(\cdot;\theta^c)$ and $D$ is the batch size. By viewing the unknown distributional parameter as a random vector $\theta$ and assuming a prior distribution $\pi_0$ on $\theta$, the posterior distribution of $\theta$ is updated by the Bayes rule as follows:
\begin{align}\label{eq:independent-posterior}
    \pi_t(\theta) = \frac{\pi_{t-1}(\theta)f(\mathbf{y}_t;\theta)}{\int\pi_{t-1}(\theta)f(\mathbf{y}_t;\theta)d\theta} =  \frac{\pi_{t-1}(\theta) \prod_{j=1}^{D} f(y_{t,j};\theta)}{\int\pi_{t-1}(\theta) \prod_{j=1}^{D} f(y_{t,j};\theta)d\theta}.
\end{align}
The objective function \eqref{eq: obj-independent} can be viewed as a function of $\theta$, so we define the following function
$$
H(x,\theta):= \mathbb{E}_{f(\cdot;\theta)}[h(x, \xi)]. 
$$
To estimate the true objective function \eqref{eq: obj-independent}, we consider the Bayesian average of the objective function:
\begin{align}
    \min_{x \in \mathcal{X}} \mathbb{E}_{\pi_{t}}\left[H(x, \theta)\right],
\label{eq: independent}
\end{align} 
where the expectation is taken w.r.t. the posterior distribution $\pi_t$ defined in \eqref{eq:independent-posterior}. Then we apply SGD on \eqref{eq: independent} for $K$ iterations within each time stage, where $K$ is a user choice or limited by the time length of the current stage before the next batch of data come in. The key element in SGD is the stochastic gradient estimator, and an unbiased gradient estimator of the objective function in \eqref{eq: independent} can be computed by the infinitesimal perturbation analysis (IPA, refer to \cite{Fu2008}) as:
\begin{align}
    \nabla_x h(x, \xi),  ~~\xi \sim f(\cdot;\theta)~ \text{and}~ \theta \sim \pi_t.
\label{eq: independent gradient estimator}
\end{align}

Now let's focus on the decision-dependent case. With slight abuse of notations, we use the same notations as in the decision-independent case unless defined otherwise. Unlike the decision-independent case where the data batches are i.i.d. over time from the fixed distribution $f(\cdot;\theta^c)$, in the decision-dependent case data batches $\{\mathbf{y}_t\}_t$ are correlated and differently distributed across time stages, since $\mathbf{y}_t$ depends on the decision $x_t$ which is in turn updated from previous data over time. Regardless of the non-stationarity of the data batches, we still use Bayesian posterior distribution to estimate $\theta$:
\begin{align}\label{eq:dependent-posterior}
    \pi_{t}(\theta) = \frac{\pi_{t-1}(\theta) f\left(\mathbf{y}_{t}; x_{t}, \theta \right)}{\int \pi_{t-1}(\theta) f\left(\mathbf{y}_{t}; x_{t}, \theta \right)d\theta} =  \frac{\pi_{t-1}(\theta) \prod_{j=1}^{D} f(y_{t,j};x_t,\theta)}{\int\pi_{t-1}(\theta) \prod_{j=1}^{D} f(y_{t,j};x_t,\theta)d\theta}.
\end{align}

Due to the nonstationarity of data batches, the consistency of the posterior distribution is a question here; we will characterize the conditions needed for strong consistency of $\pi_t$ in \cref{sec: convergence}. 
The Bayesian average of the objective function is
\begin{align}
    \mathbb{E}_{\pi_{t}}[H(x, \theta)] = \mathbb{E}_{\pi_{t}}\left[\mathbb{E}_{f\left(\cdot; x,\theta \right)} \left[h\left(x, \xi\right)\right]\right].
\label{eq: dependent}
\end{align} 
An unbiased gradient estimator of the objective function \cref{eq: dependent} is 
\begin{align}
    \nabla_{x} h(x, \xi) + h(x, \xi)\frac{\nabla_x\widehat{f}_{t}\left(\xi; x\right)}{ \widehat{f}_{t}\left(\xi ; x\right)},  ~~\xi \sim f(\cdot;\theta)~ \text{and}~ \theta \sim \pi_t, 
\label{eq: dependent gradient estimator}
\end{align}
where $\hat{f}_t(\cdot;x):=\mathbb{E}_{\pi_t}[f(\cdot;x,\theta)]$, $\nabla_x \hat{f}_t(\cdot;x):=\nabla_x \mathbb{E}_{\pi_t}[f(\cdot;x,\theta)]$. The derivation of the gradient estimators \eqref{eq: independent gradient estimator} and \eqref{eq: dependent gradient estimator} will be shown in \cref{sec: convergence}. Informally, \eqref{eq: dependent gradient estimator} is obtained by taking derivative of $h(x,\xi)f(\xi;x,\theta)$ w.r.t. $x$. In the algorithms we assume that the posterior distribution $\pi_t$ and the expectation in $\hat{f}_t(\cdot;x)$ and $\nabla_x \hat{f}_t(\cdot;x)$ can be exactly computed, which is often the case when we choose a conjugate prior distribution for Bayesian updating. For general posterior distributions, we can use general Markov Chain Monte Carlo (MCMC) methods, such as the Langevin algorithm (\cite{ermak1980numerical, durmus2017fast}), to sample from the posterior and use these samples to approximate the expectation. It is worth noting that the first term in \eqref{eq: dependent gradient estimator} is the same as the stochastic gradient estimator \eqref{eq: independent gradient estimator} in the decision-independent case, and the second term is unique here and caused by the dependence of the distribution on the decision $x$. 

The algorithms, named as Bayesian Stochastic Gradient Descent (Bayesian-SGD), for stochastic optimization with decision-independent uncertainty and decision-dependent uncertainty are shown in \cref{algorithm: independent} and \cref{algorithm: dependent}, respectively. Please note that to accelerate algorithm convergence, variants of SGD methods could be used instead of the plain SGD iterations in these algorithms.

\begin{algorithm}
\caption{Bayesian-SGD (decision-independent uncertainty)}
\label{algorithm: independent}
\begin{algorithmic}
\STATE{\textbf{input}: data batch size $D$, number of SGD iterations $K$, step size sequence $\{a_{t,j}, t=1,2,\ldots; j=0,\ldots,K-1\}$, time horizon $T$.}
\STATE{\textbf{initialization}: choose an initial decision $x_1$ and prior distribution $\pi_0(\theta)$.}
\FOR{$t =1: T$}
    \STATE{-A batch of data $y_{t,1},\cdots,y_{t,D} \iid f(\cdot; \theta^c)$ arrives;}
    \STATE{-Posterior Update: compute $\pi_t(\theta)$ according to \eqref{eq:independent-posterior}. }
    \STATE{-Decision Update:
        \begin{itemize}
            \item set $x_{t,0}:=x_{t}$;
            \item for $j=0,\cdots,K-1$, draw sample $\theta_{t,j} \sim \pi_t(\theta)$ and $\xi_{t,j} \sim f(\cdot;\theta_{t,j})$, and carry out SGD iteration:
        \end{itemize}
        \begin{align}\label{eq: SGD_independent}
            x_{t,j+1}:=\operatorname{Proj}_{\mathcal{X}}\left\{x_{t,j}-a_{t,j}\nabla_x h(x_{t,j},\xi_{t,j})\right\},
        \end{align}
        ~~~~~~~~~~~~~~where $\operatorname{Proj}_{\mathcal{X}}$ is a projection operator that projects the iterate  to the  set $\mathcal{X}$. 
        \begin{itemize}
            \item set the updated decision as $x_{t+1}:=x_{t,K}$;
        \end{itemize}
    }
\ENDFOR
\RETURN $x_{T+1}$
\end{algorithmic}
\end{algorithm}
    
\begin{algorithm}
\caption{Bayesian-SGD (decision-dependent uncertainty)}
\label{algorithm: dependent}
\begin{algorithmic}
\STATE{\textbf{input}: data batch size $D$, number of SGD iterations $K$, step size sequence $\{a_{t,j}, t=1,2,\ldots; j=0,\ldots,K-1\}$, time horizon $T$.}
\STATE{\textbf{initialization}: choose an initial decision $x_1$ and prior distribution $\pi_0(\theta)$.}
\FOR{$t=1:T$}  
    \STATE{-A batch of data $y_{t,1},\cdots,y_{t,D} \iid f(\cdot; x_t, \theta^c)$ arrives;}
    \STATE{-Posterior Update: compute $\pi_t(\theta)$ according to \eqref{eq:dependent-posterior}. }
    \STATE{-Decision Update: 
        \begin{itemize}
            \item set $x_{t,0}:=x_{t}$;
            \item for $j=0,\cdots,K-1$, draw sample $\theta_{t,j} \sim \pi_t(\theta)$ and $\xi_{t,j} \sim f(\cdot;x_{t,j},\theta_{t,j})$, and carry out SGD iteration:
        \end{itemize}
        \begin{align}\label{eq: SGD_dependent}
            x_{t,j+1}:=\operatorname{Proj}_{\mathcal{X}}\left\{x_{t,j}-a_{t,j}\left(\nabla_{x} h\left(x_{t,j}, \xi_{t,j}\right)+h\left(x_{t,j}, \xi_{t,j}\right) \frac{\nabla_x\widehat{f}_{t}\left(\xi_{t,j} ; x_{t,j}\right)}{ \widehat{f}_{t}\left(\xi_{t,j} ;x_{t,j}\right)}\right)\right\},
        \end{align}
        ~~~~~~~~~~~~~~where $\operatorname{Proj}_{\mathcal{X}}$ is a projection operator that projects the iterate to the set $\mathcal{X}$. 
        \begin{itemize}
            \item set the updated decision as $x_{t+1}:=x_{t,K}$;
        \end{itemize}
    }
\ENDFOR
\RETURN $x_{T+1}$
\end{algorithmic}
\end{algorithm}

\section{Convergence analysis}
\label{sec: convergence}
In this section, we show asymptotic convergence of Algorithm 2.1 and Algorithm 2.2. Towards this end, we first need to show the consistency of the Bayesian posterior distribution and then show the convergence of SGD when applied to the non-stationary Bayesian average stochastic optimization problems \cref{eq: independent} and
\cref{eq: independent gradient estimator}. In addition, we show the convergence rate in the decision-independent case.

\subsection{Convergence analysis for the decision-independent case}
\label{subsec: independent}
Let's first consider the decision-independent case. The probability space is constructed as follows. Define the Bayesian prior $\pi_0$ on $(\Theta, \mathcal{B}_{\Theta})$, where $\mathcal{B}_{\Theta}$ is the Borel $\sigma$-algebra on $\Theta$. Let $\mathcal{Y} \subset \mathbb{R}^{m}$ denote the data (observation) space. The data $y$ takes value in $\mathcal{Y}$ equipped with a Borel $\sigma$-algebra $\mathcal{B}_{\mathcal{Y}}$ and a probability measure $\left\{\mathbb{P}_{\theta^c}\right\}$, such that $\mathbb{P}_{\theta^c}(y \in A) = \int_{A} f\left(y; \theta^c\right) d y, \forall A \in \mathcal{B}(\mathcal{Y})$. For the sequence $y_1,y_2,\cdots,y_n \iid f(\cdot;\theta^c)$, the probability measure is denoted by $\mathbb{P}_{\theta^c}^{n}$. As for the infinite sequence $\{y_1,y_2,\ldots\}$, the probability measure $\mathbb{P}_{\theta^c}^{\infty}$ can be constructed by Kolmogorov's extension theorem (cf. Theorem A.3.1 in \cite{durrett2019probability}). In the following, w.p.1 (or almost surely) means that the considered property holds with probability one w.r.t. the probability measure $\mathbb{P}_{\theta^c}^\infty$. Finally, let $\mathcal{F}_{t}:=\sigma\left\{\left(y_{\tau} \right), \tau \leq t\right\}$ be the $\sigma$-filtration generated by the data. We have the convergence of the posterior distribution $\{\pi_t\}$ that is updated according to \cref{eq:independent-posterior} under the following assumptions.

\begin{assumption}[\cite{shapiro2021bayesian}, Assumption 3.1] (i) The set $\Theta$ is convex and compact with non-empty interior. (ii) $\ln \pi_0(\theta)$ is bounded on $\Theta$. (iii) $f(\xi|\theta) > 0$ for all $\xi \in \Xi$ and $\theta \in \Theta$. (iv) $f(\xi|\theta)$ is continuous in $\theta \in \Theta$. (v) $\ln f(\xi|\theta), \theta \in \Theta$ is dominated by an integrable (w.r.t. $\xi \sim f(\cdot;\theta^c)$) function. (vi) The data batches are i.i.d. over time from the fixed distribution $f(\cdot;\theta^c)$. 
\label{ass: Alex2022}
\end{assumption}

We refer the readers to \cite{shapiro2021bayesian} for detailed explanations of the above assumptions. The next lemma shows the Bayesian consistency under \cref{ass: Alex2022}, which implies the distributional uncertainty diminishes as $t \to \infty$. 

\begin{definition}[Weak convergence]
A sequence of distributions $\mathbb{P}_n \Rightarrow P$, if and only if $\int g d\mathbb{P}_{n}\rightarrow \int g d \mathbb{P}$ as $n \to \infty$ for all $g$ bounded and continuous. 
\end{definition}

\begin{lemma}[\cite{shapiro2021bayesian}, Lemma 3.2]
Under \cref{ass: Alex2022}, $\pi_t(\theta) \Rightarrow \delta_{\theta^c}(\theta)$ w.p.1, where $\delta_{\theta^c}$ is the Dirac delta function concentrated on the true parameter $\theta^c$.
\label{lemma: Alex2022}
\end{lemma}

We then study the asymptotic behavior of \cref{algorithm: independent} by the ordinary differential equation (ODE) method (please refer to \cite{Kushner2003} for a detailed exposition on the ODE method for stochastic approximation). The main idea is that SGD can be viewed as a noisy discretization of an ODE. Under certain conditions, the noise in SGD averages out asymptotically, such that the SGD iterates converge to the solution trajectory of the ODE. For simplicity, we consider the case where $K=1$ and rewrite the SGD iteration \cref{eq: SGD_independent} as
\begin{equation}\label{eq: SGD_independent_2}
    x_{t+1}=x_{t}-a_{t} \nabla_{x} h\left(x_{t}, \xi_{t}\right) + a_t z_t,
\end{equation}
where $a_t z_t$ is the projection term, i.e., the vector of shortest Euclidean length needed to keep the decision $x_{t+1}$ from leaving the decision space $\mathcal{X}$. We first show that under certain mild conditions, the proposed gradient estimator in \cref{eq: SGD_independent_2} is unbiased. 

\begin{assumption}\label{ass: independent_DCT}
$h(x,\xi)$ is $C^{1}$-smooth in $x$ for all $\xi \in \Xi$, and the map $\xi \to \nabla_x h(x,\xi)$ is $L_h$-Lipschitz continuous for any $x \in \mathcal{X}$. 
\end{assumption}

\cref{ass: independent_DCT} is a commonly used smooth assumption in the stochastic approximation literature (cf. \cite{ghadimi2013stochastic, Dmitriy2020}). An important consequence is that for any probability measure, $\mathbb{E} h(x,\xi)$ is differentiable in $x$ with gradient $\mathbb{E} \nabla_x h(x,\xi)$ (cf. \cite{Dmitriy2020}).

\begin{lemma}\label{lemma: independent_unbiasedness}
Under \cref{ass: independent_DCT}, $\nabla_{x} h(x, \xi)$ with $\xi \sim f(\cdot;\theta)$ and $\theta \sim \pi_t$ is an unbiased gradient estimator of the objective function in \cref{eq: independent}.
\end{lemma}

\begin{proof}
For every fixed $x\in\mathcal{X}$, 
\begin{align*}
    \mathbb{E}_{\pi_{t}}\left[\mathbb{E}_{f(\cdot;\theta)} [\nabla_x h(x, \xi)]\right] & = \mathbb{E}_{\pi_{t}}\left[\nabla_{x} \mathbb{E}_{f(\cdot;\theta)} [h(x, \xi)]\right]\\
    & = \nabla_{x} \mathbb{E}_{\pi_{t}}\left[\mathbb{E}_{f(\cdot;\theta)} [h(x, \xi)]\right],
\end{align*}
where the first equality holds because the gradient $\nabla_x h(x,\xi)$ is Lipschitz continuous, and the interchange between expectation and differentiation is justified by dominated convergence theorem (DCT). Similarly, the second equality above is again justified by DCT. Therefore, the proposed estimator in \cref{eq: SGD_independent_2} is unbiased gradient estimator of the objective function in \cref{eq: independent}.
\end{proof}

\begin{assumption}\label{ass: independent_step}
~
\begin{itemize}
    \item The step size $\{a_t\}$ satisfies $\sum_{t=1}^{\infty} a_{t}^{2}<\infty$, $\sum_{t=1}^{\infty} a_{t}=\infty$, $\lim_{t \to \infty} a_{t} = 0$, $a_t > 0, \forall t > 0$.
    \item The decision space $\mathcal{X} \subset \mathbb{R}^{d}$ is compact and convex.
\end{itemize}
\end{assumption}

The above assumptions on the step size and the compact and convex decision space are often used in SGD (cf. \cite{Kushner2003}). The first assumption essentially requires the step size diminishes to zero not too slow ($\sum_{t=1}^{\infty} a_{t}^{2}<\infty$) nor too fast ($\sum_{t=1}^{\infty} a_{t}=\infty$). For example, we can choose $a_t=\frac{a}{t}$ for some $a>0$. 

Before proceeding to our main convergence result, we introduce the continuous-time interpolations of the decision sequence $\{x_t\}$. Define $t_1=1$ and $t_n=1+\sum_{i=1}^{n-1}a_i, n \geq 2$. For $t \geq 1$, let $N(t)$ be the unique $n$ such that $t_n \leq t < t_{n+1}$. For $t < 1$, set $N(t)=1$. Define the interpolated continuous process $X$ as $X(1) = x_1$ and $X(t) = x_{N(t)}$ for any $t > 1$, and the shifted process as $X^{n}(s) = X(s + t_n)$. 
% as follows. When $\mathcal{X}$ is a hyperrectangle, for $X \in \mathcal{X}^0$ , the interior of $\mathcal{X}$, $\mathcal{C}(X)$ contains only the zero element; for $X \in \alpha \mathcal{X}$, the boundary of $\mathcal{X}$, let $\mathcal{C}(X)$ be the infinite convex cone generated by the outer normals at $X$ of the faces on which $X$ lies. For other more general compact spaces, we refer the readers to Chapter 4.3 in \cite{Kushner2003} for the construction of set $\mathcal{C}(X)$.
We then show in the following theorem that \cref{algorithm: independent} converges w.p.1.

\begin{theorem}
Let $\mathcal{D}^d[0,\infty)$ be the space of $\mathbb{R}^d$-valued operators which are right continuous and have left-hand limits for each dimension. Under \cref{ass: Alex2022}, \cref{ass: independent_DCT} and \cref{ass: independent_step}, there exists a process $X^*(\cdot)$ to which the subsequence of $\{X^n(\cdot)\}_n$ converges w.p.1 in the space $\mathcal{D}^d[0,\infty)$, where $X^*(\cdot)$ satisfies the following ODE
\begin{align}\label{ODE_independent}
\dot{X}= -\nabla H(X,\theta^c)+z, ~z\in -\mathcal{C}(X), \quad X(1)=x_1,
\end{align}
where $\mathcal{C}(X)$ is the Clarke's normal cone to $\mathcal{X}$, i.e., for any $x \in \mathcal{X}$, $\mathcal{C}(x) = \{c: c^{T}x \geq c^{T}y, \forall y \in \mathcal{C}\}$. $z$ is the projection term: it is the vector of shortest Euclidean length needed to keep the trajectory of the ODE $X(\cdot)$ from leaving the decision space $\mathcal{X}$. The sequence $\{x_t\}_t$ in \cref{eq: SGD_independent_2} also converges w.p.1 to the limit set of the ODE \cref{ODE_independent}.
\label{thm: independent}
\end{theorem}

\begin{proof}
Note that 
\begin{align*}
    & \mathbb{E}\left[ \nabla_{x} h\left(x_{t}, \xi_{t}\right)|x_1, y_s, \xi_s, s<t \right] \\
    &= \mathbb{E}_{\pi_{t}}[\mathbb{E}_{f(\cdot;\theta)}\left[\nabla_{x} h\left(x_{t}, \xi\right)\right]]\\
    &= \nabla_x H(x_t, \theta^c) +\left(\mathbb{E}_{\pi_{t}}[\mathbb{E}_{f(\cdot;\theta)}\left[\nabla_{x} h\left(x_{t}, \xi\right)\right]] -\nabla_x H(x_t, \theta^c)\right)\\
    &= \nabla_x H(x_t, \theta^c) +\left(\mathbb{E}_{\pi_{t}}[\mathbb{E}_{f(\cdot;\theta)}\left[\nabla_{x} h\left(x_{t}, \xi\right)\right]] -\mathbb{E}_{\delta_{\theta^c}}[\mathbb{E}_{f(\cdot;\theta)}\left[\nabla_{x} h\left(x_{t}, \xi\right)\right]]\right).
\end{align*}
Let $\epsilon_t =\mathbb{E}_{\pi_{t}}[\mathbb{E}_{f(\cdot;\theta)}\left[\nabla_{x} h\left(x_{t}, \xi\right)\right]] -\mathbb{E}_{\delta_{\theta^c}}[\mathbb{E}_{f(\cdot;\theta)}\left[\nabla_{x} h\left(x_{t}, \xi\right)\right]]$. By \cref{lemma: Alex2022}, $\pi_t(\theta) \Rightarrow \delta_{\theta^c}(\theta)$ w.p.1 ($\mathbb{P}_{\theta^c}^{\infty}$), and by Theorem 3.1 in \cite{Wu2018}, $\epsilon_t \rightarrow 0$ as $t \to \infty$ w.p.1 ($\mathbb{P}_{\theta^c}^{\infty}$). We can then directly apply Theorem 5.2.3 in \cite{Kushner2003} and obtain the result.
\end{proof}

\begin{remark}
The SGD iterates specified in \eqref{eq: SGD_independent} approach the solution trajectory of the ODE \eqref{ODE_independent} and eventually converges to a limit point of the ODE, which is a point $x^{*}$ satisfying $\nabla H(x^{*},\theta^c)=0$ if the point is in the interior of $\mathcal{X}$. Hence, such a point is a stationary point of problem \eqref{eq: obj-independent} for the decision-independent case and can be a local optimal solution if it is stable. On a related note, stochastic gradient Langevin dynamic (SGLD), a popular variant of SGD, adds properly scaled isotropic Gaussian noise to an unbiased estimate of the gradient at each iteration, which allows the solution trajectory to escape local minimum and guarantees asymptotic convergence to a global minimizer for sufficiently regular non-convex objectives (see \cite{raginsky2017non, durmus2019analysis} and references therein). It is an interesting future direction to apply SGLD to our considered stochastic optimization problem with streaming input data.
\end{remark}

Next, we investigate the convergence rate of \cref{algorithm: independent} for the unconstrained case, i.e., without the projection term $a_t z_t$ under the following additional assumptions.

\begin{assumption}\label{ass: independent_rate}
~
\begin{itemize}
    \item The parameter space $\Theta$ is finite, i.e., $\Theta = \{\theta_1,\ldots,\theta_k\}$. Moreover, $\theta^c \in \Theta$.
    \item There exists $0<L_{H}<\infty$ such that $||\nabla_x H(x,\theta_1)-\nabla_x H(x,\theta_2)||_2 \leq L_{H}||\theta_1-\theta_2||_2$ for all $\theta_1,\theta_2 \in \Theta$ and for all $x \in \mathcal{X}$.
    \item Sampling variance is bounded by $\sigma^2$, i.e., $\mathbb{E}[||\nabla_x h(x,\xi) - \nabla_x H(x,\theta)||_2^2|\theta] \leq \sigma^2$, for all $\theta \in \Theta$.
\end{itemize}
\end{assumption}

Due to technical challenges, in \cref{ass: independent_rate} we only consider a finite parameter space, which is practical in many real-world problems. For example, it can be viewed as a discrete approximation of a continuous parameter set, and the discretization can be chosen of any precision. The second assumption essentially requires $H(x,\theta)$ is $C^{1}$-smooth in $\theta$ for all $x \in \mathcal{X}$ and is a common assumption in stochastic approximation literature (cf. \cite{Song2019a}). The bounded sampling variance is also a common assumption in non-convex SGD convergence analysis (cf. \cite{shamir2013stochastic}). 

Under \cref{ass: independent_rate}, we can show the bias term (the difference between $\mathbb{E}_{\pi_t} \nabla_x H(x,\theta)$ and $\mathbb{E}_{\pi_t} \nabla_x H(x,\theta^c)$) can be upper bounded with high probability, which serves as a key lemma in showing the convergence rate of the decision-independent algorithm. 

\begin{lemma}\label{lemma: independent_concentration}
Under \cref{ass: independent_rate}, there exists a constant $C_1>0$ such that for any $\delta>0$, with probability at least $1-\delta$ we have
\begin{align*}
	||\mathbb{E}_{\pi_t}\nabla_x H(x, \theta)-\mathbb{E}_{\pi_t}\nabla_x H(x, \theta^c)||_2^2\leq C_1 \frac{\log Dt + \log \frac{1}{\delta}}{Dt}, \forall x\in \mathcal{X}, \forall t>0. 
\end{align*}
\end{lemma}

The proof of \cref{lemma: independent_concentration} can be found in \cref{appendix: A}. Next, we show the convergence rate of \cref{algorithm: independent}. To simplify the analysis and also be consistent with the convergence analysis of smooth non-convex SGD, we consider a variant of SGD where the final output is randomly chosen as follows: let $z_T = x_{t}$ with probability $\frac{a_t}{\sum_{t=1}^{T}a_t}$, $t=1,\cdots,T$. The randomization scheme helps with the analysis of the expected gradient of the final output under the true parameter $\theta^c$, and has been widely used in the smooth non-convex SGD literature (cf. \cite{ghadimi2013stochastic}). We then have the following theorem giving the convergence rate of the randomized output algorithm under different step sizes.

\begin{theorem}\label{thm: independent_rate}
Under \cref{ass: Alex2022}, \cref{ass: independent_DCT}, \cref{ass: independent_step}, and \cref{ass: independent_rate}, for any $\delta>0$, we have with probability at least $1-\delta,$ for any $T>0$, the following bound on the expected gradient of the final output under the true parameter $\theta^c$  
\begin{enumerate}
    \item[(i)] If the step size satisfies $a_t=\frac{a}{\sqrt{T}}$, $\forall t \leq T$, for some constant $a<\frac{\sqrt{T}}{L_h}$, then 
    {\small
    \begin{align*}
        & \mathbb{E}[\|\nabla_x H(z_T,\theta^c) \|_2^2] \\
	    &\leq \left[\frac{2(H(x_1,\theta^c)- \min_{x\in\mathcal{X}}H(x,\theta^c))}{a\sqrt{T}}\right] + \left[\frac{A_1}{T} + \frac{A_2 \log T}{T} + \frac{A_3 \log^2 T}{T}\right] + \frac{L_h a\sigma^2}{\sqrt{T}},
    \end{align*}
    }
    where $A_1=\frac{C_1 (\log D - \log \delta)}{L_h D}$, $A_2=\frac{C_1 (\log D - \log \delta)}{L_h D} + \frac{C_1}{L_h D}$, $A_3=\frac{C_1}{L_h D}$.
    \item[(ii)]  If the step size satisfies $a_t = \frac{a}{t}$, $\forall t \leq T$, for some constant $a < \frac{1}{L_h}$, then
    {\small
    \begin{align*}
	& \mathbb{E}[\|\nabla_x H(z_T,\theta^c) \|_2^2]\\
	&\leq \left[\frac{2(H(x_1,\theta^c)- \min_{x\in\mathcal{X}}H(x,\theta^c))}{a} +\frac{6C_1 + \pi^2C_1 (\log D - \log \delta)}{6D}+\frac{\pi^2 L_h a \sigma^2}{6}\right]\frac{1}{\log T}.
    \end{align*}
    }
    \item[(iii)]  If the step size satisfies $a_t = \frac{a}{\sqrt{t}}$, $\forall t \leq T$, for some constant $a < \frac{1}{L_h}$, then
    {\small
    \begin{align*}
	& \mathbb{E}[\|\nabla_x H(z_T,\theta^c) \|_2^2]\\
	&\leq [\frac{2(H(x_1,\theta^c)- \min_{x\in\mathcal{X}}H(x,\theta^c))}{a\sqrt{T}} + \frac{3C_1(\log D - \log \delta) + 4C_1}{D\sqrt{T}} + \frac{L_h a \sigma^2}{\sqrt{T}}] + \frac{L_h a \sigma^2 \log T}{\sqrt{T}}.
    \end{align*}
    }
\end{enumerate}
\end{theorem}

The proof of \cref{thm: independent_rate} can be found in \cref{appendix: B}.
\cref{thm: independent_rate} shows that for the constant step size $a_{t}=\frac{a}{\sqrt{T}}$, the convergence rate is $O(\frac{1}{\sqrt{T}})$. Note that in case (i), the first term in the convergence rate depends on the initialization of the solution (difference between $H(x_1,\theta^c)$ and $\min_x H(x,\theta^c)$); the last term depends on the Lipschitz constant and sampling variance. These two terms are consistent with the classical smooth non-convex SGD (cf. \cite{ghadimi2013stochastic}). The second, third, and fourth terms are caused by the difference between $\mathbb{E}_{\pi_t} \nabla_x H(x,\theta)$ and $\mathbb{E}_{\pi_t} \nabla_x H(x,\theta^c)$, which is due to the Bayesian estimation that is unique to the considered problem. As for the classical decreasing step size $a_{t} = \frac{a}{t}$, the convergence rate is $O(1 / \log T)$. For the bigger decreasing step size $a_{t} = \frac{a}{\sqrt{t}}$, the convergence rate is $O(\log T / \sqrt{T})$. 

\subsection{Convergence analysis for the decision-dependent case}
\label{subsec: dependent}
In this section, we theoretically study the convergence behavior of \cref{algorithm: dependent}. We follow the approach in \cite{cutler2022stochastic} to construct the probability space for the decision-dependent case. Note that the data $y$ takes value in the space $\mathcal{Y}$ equipped with a Borel $\sigma$-algebra $\mathcal{B}_{\mathcal{Y}}$ and a probability measure $\mathbb{P}_{\theta^c}(\cdot|x)$ such that $\mathbb{P}_{\theta^c}(y \in A|x) = \int_A f(y;x,\theta^c)dy, \forall A \in \mathcal{B}_{\mathcal{Y}}$. Suppose that there is a probability space $(\mathcal{S}, \mathcal{H}, \mu)$ and a measurable map $F: \mathcal{S}\times \mathcal{X}\rightarrow \mathcal{Y}$ such that for every  set $A \in \mathcal{B}_{\mathcal{Y}}$, the $\mathbb{P}_{\theta^c}(\cdot|x)$-measure of $A$ is equal to the $\mu$-measure of the set $\{s\in\mathcal{S}: F(s,x)\in A\}$. Then we define $(\Omega, \mathcal{F}, \mathbb{P}_{\theta^c}^\infty)$ as the countable product $(\mathcal{S}, \mathcal{H},\mu)^\infty$. In the following, w.p.1 (or almost surely) means that the considered property holds with probability one w.r.t. the probability measure $\mathbb{P}_{\theta^c}^{\infty}$. Let $\mathcal{F}_t=\sigma\{(x_{\tau},y_{\tau}), \tau \leq t\}$ be the $\sigma$-filtration generated by the data and decision sequences. For simplicity, we assume at each time stage the data batch size $D=1$ and the number of SGD iterations $K=1$. We have the convergence of the posterior distribution $\{\pi_t\}$ that is updated according to \cref{eq:dependent-posterior} under the following assumptions.

\begin{assumption}
~
\begin{itemize}
    \item The parameter space $\Theta$ is discrete. Moreover, $\theta^c \in \Theta$.
    \item The prior distribution $\pi_0(\theta^c) > 0$.
\end{itemize}
\label{ass: consistency}
\end{assumption}

The assumptions above are regularity conditions and easy to be verified in practice. Note that \cref{algorithm: dependent} works for a general parameter space, but due to technical challenges, we assume a discrete parameter space for the convergence analysis. Note that for the decision-dependent case, the correlated and differently distributed data $\{\mathbf{y}_t\}$ pose a great challenge to analyzing the consistency of the Bayesian posterior distribution $\pi_t$. To prove the Bayesian consistency, we first show the following intermediate result. Let $D_{K L}(P\|Q):=\int \log \left(\frac{d P}{d Q}\right) d P$ denote the Kullback-Leibler (K-L) divergence from distribution $P$ to distribution $Q$.

\begin{lemma}
Suppose \cref{ass: consistency} holds. Recall $\hat f_t(\cdot;x) = \sum_\theta\pi_t(\theta) f(\cdot;x,\theta)$. Denote $ f^*(\cdot;x) :=  f(\cdot;x,\theta^c)$, for any $x \in \mathcal{X}.$  At decision $x_{t+1},$ the K-L divergence from $f^*$ to $\hat f_t$ is denoted as $d_t,$ i.e., $d_t :=D_{KL}(f^*(\cdot;x_{t+1})||\hat f_t(\cdot;x_{t+1})).$ Then we have $$\lim_{t\rightarrow \infty} d_t = 0~~\text{and}~~ \sum_{t=1}^\infty d_t<\infty,  ~\text{w.p.1} (\mathbb{P}_{\theta^c}^{\infty}).$$ 
\label{lemma: consistency}
\end{lemma}

The proof of \cref{lemma: consistency} can be found in \cref{appendix: D}. Intuitively, \cref{lemma: consistency} implies that with more observation data even at different decisions, we know more about the true parameter $\theta^c$ and are able to provide a more precise estimation of the density $f^*$ at the next decision. Moreover, if we know that each $\theta$ is identifiable as rigorously defined in the following assumption, we can further prove the consistency of $\{\pi_t\}$ regardless of the correlation and non-stationarity of the observation data.

\begin{assumption}[Linear Independence]
\label{ass:linear_independence}
For almost every $x$ in $\mathcal{X}$, for any $\mathcal{K} \subseteq \mathbb{N}$ where $\mathbb{N}$ is the set of natural numbers, $\{f(\cdot;x,\theta_i)\}_{i \in \mathcal{K}}$ are linearly independent in $\mathcal{Y},$ i.e., 
$$ \sum_{i \in \mathcal{K}} c_i f(y;x,\theta_i)=0,~~ \forall y \in\mathcal{Y}~~\Rightarrow ~~c_i=0 ~~ \forall i \in \mathcal{K}.$$ 
\end{assumption}

\cref{ass:linear_independence} intuitively requires that for almost every decision $x$, the observation distributions generated from different $\theta$'s are distinguishable (or identifiable, cf. Definition 5.2 in \cite{lehmann2006theory}). For the ease of notation, we denote the density function as $f(\cdot;x,\theta):=f(\cdot;g(x,\theta))$, where $g: \mathbb{R}^{d} \times \mathbb{R}^{l} \to \mathbb{R}^{s}$ is a mapping from $\mathcal{X} \times \Theta$ to the $s$-dimensional parameter space of the distribution. A necessary condition for \cref{ass:linear_independence} to hold is: $g(x,\theta_1) \neq g(x,\theta_2)$ for almost every $x \in \mathcal{X}$ and for all $\theta_1 \in \Theta$, $\theta_2 \in \Theta$ such that $\theta_1 \neq \theta_2$. Under this necessary condition, \cref{ass:linear_independence} is satisfied by many distributions families. For example, the Wronskian Determinant for exponential distributions with different parameters $g(x,\theta_1), \cdots, g(x,\theta_n)$ is computed as $W(\xi)=\prod_{i=1}^{n}g(x,\theta_i) \exp(-\sum_{i=1}^{n} g(x,\theta_i) \xi) \prod_{i \neq j}(g(x,\theta_i)$
$-g(x,\theta_j))$, which is nonzero for almost every $x \in \mathcal{X}$ and all $\xi \in \Xi$ when $\theta_i$'s are distinct, which directly implies the linear independence of $\{f(\cdot;x,\theta_i)\}_i$. For other exponential families, such as normal, gamma, and Poisson, a general solution to check the Wronskian Determinant may not be readily available. Instead, one could check whether the components of the sufficient statistics are linearly independent, i.e., whether the exponential family is minimal (cf. Chapter 1.5 in \cite{lehmann2006theory}).

\begin{assumption}\label{ass: compact space}
The decision space $\mathcal{X} \subset \mathbb{R}^{d}$ is compact and convex.
\end{assumption}

We then have the following proposition on the consistency of the posterior distribution $\{\pi_t\}$.

\begin{proposition}
Under \cref{ass: consistency}, \cref{ass:linear_independence} and \cref{ass: compact space},  $\pi_t \Rightarrow \delta_{\theta^c}$ w.p.1 ($\mathbb{P}_{\theta^c}^{\infty}$).
\label{lemma: pi_consistency}
\end{proposition}

The proof of \cref{lemma: pi_consistency} can be found in \cref{appendix: E}. \cref{lemma: pi_consistency} guarantees that although the observation at each time depends on the current decision, it can provide enough information to ensure the posterior distribution will eventually concentrate on the true parameter. In the following, we will show that the consistency of $\{\pi_t\}$ ensures that the gradient estimator is accurate enough and thus \cref{algorithm: dependent}  converges.

\begin{remark}
We note that the consistency of posterior distributions for non i.i.d. observations is previously shown in \cite{ghosal2007convergence}. However, they give very general convergence result with assumptions (such as existence of testing function sequence) that are often abstract  and hard to verify in practice. On the other hand, our Bayesian consistency result is built on assumptions (in particular \cref{ass:linear_independence}) that are easy to verify and interpret.
\end{remark}

We then study the asymptotic behavior of \cref{algorithm: dependent} by the ODE method similar to the decision-independent case. We can rewrite the SGD iteration \cref{eq: SGD_dependent} as
\begin{equation}\label{eq: SGD_dependent_2}
    x_{t+1} = x_t - a_t \left(\nabla_x h(x_t,\xi_t)+h(x_t,\xi_t)\frac{
       \nabla_x\widehat{f}_{t}\left(\xi_{t} ; x_{t}\right)}{ \widehat{f}_{t}\left(\xi_{t} ; x_{t}\right)}\right)+a_t z_t,
\end{equation}
where $a_t z_t$ is the projection term. We show that under certain mild conditions, the proposed gradient estimator \cref{eq: SGD_dependent_2} is unbiased. 

\begin{assumption}\label{ass: dependent_DCT}
The density function $f(\xi;x,\theta)$ is $C^{1}$-smooth in $x$ for all $\xi \in \Xi$ and for all $\theta \in \Theta$.
\end{assumption}

Together with \cref{ass: independent_DCT}, \cref{ass: dependent_DCT} puts mild conditions that justify the interchange between differentiation and integral for the decision-dependent case. 

\begin{lemma}\label{lemma: dependent_unbiasedness}
Under \cref{ass: independent_DCT} and \cref{ass: dependent_DCT}, we have that $\nabla_x h(x_t,\xi)+h(x_t,\xi)\frac{\nabla_x\widehat{f}_{t}\left(\xi ; x_{t}\right)}{ \widehat{f}_{t}\left(\xi; x_{t}\right)}$ with $\xi \sim f(\cdot;x_t, \theta)$ and $\theta \sim \pi_t$ is an unbiased gradient estimator of the objective function in \cref{eq: dependent}.
\end{lemma}

The detailed derivation can be found in \cref{appendix: C}. Note that in performative prediction literature (e.g. \cite{Dmitriy2020}), the gradient estimator is also derived using the chain rule similar to \cref{eq: dependent gradient estimator}. However, due to the difficulty in estimating the second term, most of the literature in performative prediction focus only on the first term, and show that under the biased gradient estimator, the solution converges to a so-called performative stable point which is in general different from the true optimal solution. In contrast, our approach provides a Bayesian way to estimate the second term under the parametric assumption and aims to converge to the true optimal solution of problem \cref{eq: obj-dependent}.

A final set of assumption on the step size to show the convergence of \cref{algorithm: dependent} is listed below.

\begin{assumption} The step size $a_t$ satisfies $\sum_{t=1}^{\infty} a_{t}=\infty, \lim_{t \rightarrow \infty} a_{t}=0, a_{t}>0, \forall t > 0$.
\label{ass: dependent}
\end{assumption}

We then have the following theorem showing the weak convergence of \cref{algorithm: dependent}.

\begin{theorem}
Let $\mathcal{D}^d[0,\infty)$ be the space of $\mathbb{R}^d$-valued operators which are right continuous and have left-hand limits for each dimension. Under \cref{ass: independent_DCT}, \cref{ass: consistency}, \cref{ass:linear_independence}, \cref{ass: compact space}, \cref{ass: dependent_DCT} and \cref{ass: dependent}, for each subsequence of $\{X^n(\cdot)\}_n$, there exists a further subsequence  $\{X^{n_k}(\cdot)\}_{n_k}$ and a process $X^*(\cdot)$ such that $X^{n_k}(\cdot)\Rightarrow X^*(\cdot)$ in the weak sense as $t\rightarrow\infty$ in the space $D^d[0,\infty)$, where $X^*(\cdot)$ satisfies the following ODE:
\begin{align}\label{ODE_dependent}
\dot{X}= -\nabla H(X,\theta^c)+z, ~z\in -\mathcal{C}(X), \quad X(1)=x_1,
\end{align}
where $\mathcal{C}(X)$ is the Clarke's normal cone to $\mathcal{X}$, i.e., for any $x \in \mathcal{X}$, $\mathcal{C}(x) = \{c: c^{T}x \geq c^{T}y, \forall y \in \mathcal{C}\}$. $z$ is the projection term: it is the vector of shortest Euclidean length needed to keep the trajectory of the ODE $X(\cdot)$ from leaving the decision space $\mathcal{X}$. Let $L_{\mathcal{X}}$ be the set of limit points of \cref{ODE_dependent} in $\mathcal{X}.$ Then there exist $\mu_n\rightarrow 0$ and $T_n\rightarrow \infty$ such that
$$\lim_n P\left\{\sup_{t\leq T_n} \text{Dist}\left(X^n(t),L_{\mathcal{X}}\right)\geq \mu_n\right\} = 0,$$
where $\text{Dist}(x,\mathcal{E}) = \inf_{y\in \mathcal{E}} \|x-y\|_2$ for any set $\mathcal{E}$ and point $x\in \mathcal{X}.$ The sequence $\{x_t\}_t$ in \cref{eq: SGD_dependent_2} also converges weakly to the limit set of the ODE \cref{ODE_dependent}.
\label{theorem: ODE}
\end{theorem}

\begin{remark}
\cref{theorem: ODE} shows the weak convergence of \cref{algorithm: dependent}. The SGD iterates specified in \cref{eq: SGD_dependent} approaches the solution trajectory of the ODE \cref{ODE_dependent} and eventually converges to a limit point of the ODE, which is a point $x^*$ satisfying $\nabla H(x^*,\theta^c) = 0$ if the point is in the interior of $\mathcal{X}$. Hence, such a point is a stationary point of problem \cref{eq: obj-dependent} for the decision-dependent case and can be a local optimal solution if it is stable. The weak convergence result implies that once the trajectory enters the domain of attraction of a local optimal solution, the chance of escaping from it goes to 0 in the limit.
\end{remark}

Now we prove \cref{theorem: ODE} below. 
\begin{proof}
Recall that at time $t+1$, \cref{algorithm: dependent} takes the following update
\begin{equation*}
    x_{t+1} = x_t - a_t \left(\nabla_x h(x_t,\xi_t)+h(x_t,\xi_t)\frac{
       \nabla_x\widehat{f}_{t}\left(\xi_{t} ; x_{t}\right)}{ \widehat{f}_{t}\left(\xi_{t} ; x_{t}\right)}\right)+a_t z_t.
\end{equation*}
From the derivation of unbiased gradient estimator in \cref{appendix: C}, we have
\begin{align*}
    & \hspace{0.4cm} \mathbb{E}_{\pi_t}\left[\mathbb{E}_{f(\cdot;x_t,\theta)}\left[\nabla_x h(x_t,\xi)\right]\right] \\
    & = \mathbb{E}_{\hat{f}_t(\cdot;x_t)}\left[\nabla_x h(x_t,\xi)\right] \\
    &=  \mathbb{E}_{f^*(\cdot;x_t)}\nabla_x h(x_t,\xi) + \left(\mathbb{E}_{\hat f_t(\cdot;x_t)}[\nabla_x h(x_t,\xi)]- \mathbb{E}_{f^*(\cdot;x_t)}[\nabla_x h(x_t,\xi)]\right)\\
	&=  \mathbb{E}_{f^*(\cdot;x_t)}[\nabla_x h(x_t,\xi)] +\beta_{t,1},
\end{align*}
where $f^*(\cdot;x) =  f(\cdot;x,\theta^c)$ for $x \in \mathcal{X}$, $\beta_{t,1} = \mathbb{E}_{\hat f_t(\cdot;x_t)}[\nabla_x h(x_t,\xi)]- \mathbb{E}_{f^*(\cdot;x_t)}[\nabla_x h(x_t,\xi)]$. Similarly, we have
\begin{align*}
    & \hspace{0.4cm} \mathbb{E}_{\pi_t}\left[\mathbb{E}_{f(\cdot;x_t,\theta)}\left[h(x_t,\xi)\frac{\nabla_x\widehat{f}_{t}\left(\xi; x_{t}\right)}{\widehat{f}_{t}\left(\xi; x_{t}\right)}\right]\right] \\
    & = \int_\Xi h(x_t,\xi)\nabla_x\widehat{f}_{t}\left(\xi; x_{t}\right)d \xi_t\\
	& = \int_\Xi h(x_t,\xi_t)\nabla_x f^*(\xi;x_t) d \xi +\left ( \int_\Xi h(x_t,\xi)\nabla_x\widehat{f}_{t}\left(\xi; x_{t}\right) d \xi - \int_\Xi h(x_t,\xi)\nabla_x f^*(\xi;x_t) d \xi \right)\\
	& = \int_\Xi h(x_t,\xi)\nabla_x f^*(\xi;x_t) d \xi +\beta_{t,2},
\end{align*}
where $\beta_{t,2} = \int_\Xi h(x_t,\xi)\nabla_x\widehat{f}_{t}\left(\xi ; x_{t}\right) d \xi - \int_\Xi h(x_t,\xi)\nabla_x f^*(\xi;x_t) d \xi$. Note that
\begin{equation*}
    \int_\Xi h(x_t,\xi)\nabla_x f^*(\xi;x_t) d \xi + \mathbb{E}_{f^*(\cdot;x_t)}[\nabla_x h(x_t,\xi)]=\nabla_x H(x,\theta^c),
\end{equation*}
and we can rewrite the update as 
\begin{equation*}
    x_{t+1} = x_t-a_t  \nabla_x H(x_t,\theta^c) -a_t\beta_{t,1}-a_t\beta_{t,2}-a_t \delta M_t+a_t z_t,
\end{equation*}
where 
\begin{align*}
    \delta M_t = \nabla_x h(x_t,\xi_t)+h(x_t,\xi_t) \frac{\nabla_x\widehat{f}_{t}\left(\xi_t; x_{t}\right)}{ \widehat{f}_{t}\left(\xi_t; x_{t}\right)} - \mathbb{E}_{ \hat{f}_t(\cdot;x_t)} \left[\nabla_x h(x_t,\xi) + h(x_t,\xi) \frac{\nabla_x \hat{f}(\xi;x_t)}{\hat{f}(\xi;x_t)}\right]
\end{align*}
is a martingale difference sequence. Suppose that we can show $\lim_{t\rightarrow \infty} \beta_{t,1} = 0$ w.p.1 $(\mathbb{P}_{\theta^c}^{\infty})$ and $\lim_{t\rightarrow \infty} \beta_{t,2} = 0$ w.p.1 $(\mathbb{P}_{\theta^c}^{\infty})$, then the rest of the update is exactly the discretization of ODE \cref{ODE_dependent}. Then \cref{theorem: ODE} is proved by a straightforward application of Theorem 7.2.1 in \cite{Kushner2003}. We conclude the proof with the following two lemmas showing that the two bias terms $\beta_{t,1}$ and $\beta_{t,2}$ vanish in the limit. 

\begin{lemma}
Under \cref{ass: independent_DCT}, \cref{ass: consistency}, \cref{ass: dependent_DCT} and \cref{ass: dependent}, we have $\lim_{t\rightarrow \infty} \beta_{t,1} = 0$ w.p.1 $(\mathbb{P}_{\theta^c}^{\infty})$.
\label{proposition: bias vanish}
\end{lemma}

\begin{lemma}
Under \cref{ass: independent_DCT}, \cref{ass: consistency}, \cref{ass:linear_independence} and \cref{ass: dependent_DCT}, we have $\lim_{t\rightarrow \infty} \beta_{t,2} = 0$ w.p.1 $(\mathbb{P}_{\theta^c}^{\infty})$.
\label{proposition: consistency}
\end{lemma}
See \cref{appendix: F} and \cref{appendix: G} for the detailed proofs of the above two lemmas. 
\end{proof}

Finally, we summarize main similarities and differences between decision-independent and decision-dependent cases below. Both cases require a compact and convex decision space $\mathcal{X}$ and smoothness of the objective function $h(x,\xi)$ in $x$. For the decision-dependent case, we also require the density function $f(\xi;x,\theta)$ to be smooth in $x$, since the gradient estimator of the objective function in \eqref{eq: dependent} involves the gradient of $f(\xi;x,\theta)$; and moreover, we assume linear independence between densities in order to show the consistency of the posterior distribution with non i.i.d. decision-dependent data. For the decision-independent case, we further impose some stronger conditions in order to show stronger results, including the finiteness of the parameter space $\Theta$ to show the convergence rate, and stricter stepsize assumption to show the strong convergence of the solution sequence to the limit set of the ODE.

\section{Numerical experiments}
\label{sec: experiment}
\subsection{Synthetic test problems}
We first demonstrate the performance of \cref{algorithm: independent} and \cref{algorithm: dependent} on two synthetic test problems in a univariate setting and in a multivariate setting, respectively. Our method is abbreviated as Bayesian-SGD.
\subsubsection{Decision-independent uncertainty}
We first carry out numerical experiments on a simple quadratic problem in a univariate setting: $h(x,\xi)=(x-5)^2 + 0.5 \xi x$, where $\xi \sim \mathcal{N}(\theta^c,\sigma^2)$.  The parameter values are as follows: $\sigma=4$, $\theta^c=9$, $D=1$, $K=1$, $\Theta=\{1,2,\cdots,20\}$, $a_t=\frac{2}{t+5}$. It is easy to check $H(x,\theta^c) =x^2 -5.5x + 25$, and the true optimal decision is taken at $x^{*}=2.75$. At each time $t$, the gradient estimator in \cref{algorithm: independent} is $\nabla_x h(x_t,\xi_t) = 2x_t-10+0.5\xi$. In \cref{algorithm: independent}, we use the uniform distribution on $\Theta$ as the prior distribution and set the initial solution $x_1=0$. 

As a benchmark, we assume the true parameter $\theta^c$ is known and use the plain SGD algorithm on the true problem~\eqref{eq: obj-independent}. Obviously, with the knowledge of the true parameter value this algorithm should provide a lower bound on the objective value that can be achieved. We also compare with the MLE method (cf. \cite{Song2019a}), which uses the maximum likelihood estimator $\hat{\theta_t}$ at each time stage to replace the unknown $\theta^c$ in the objective function \cref{eq: obj-independent} and then solves the corresponding optimization problem by SGD.  For fair comparison, we use the same number of SGD iterations at each time stage for all three algorithms. We run all three algorithms (\cref{algorithm: independent}, benchmark, MLE) for 100 times on the problem. The mean and standard deviation of the solution error $|x_t - x^{*}|$ over time are shown in \cref{fig:decision_independent_univariate}. The observations from \cref{fig:decision_independent_univariate} can be summarized as follows.
\begin{itemize}
    \item With decreasing step size, the solution sequence in \cref{algorithm: independent} converges to the true optimal solution.
    \item The benchmark algorithm (without parameter uncertainty) performs better than the proposed algorithm and the MLE algorithm, but in the long run (e.g. $t>1000$ to be shown in multivariate setting) the three algorithms behave similarly.
    \item In the initial time stages our algorithm performs slightly better than the MLE algorithm. This is  due to the better estimation of the objective function by the Bayesian average in our algorithm than the point estimate in the MLE algorithm, when the data are limited.  
\end{itemize}

\begin{figure}[!ht]
{
\centering
\includegraphics[width=1\textwidth]{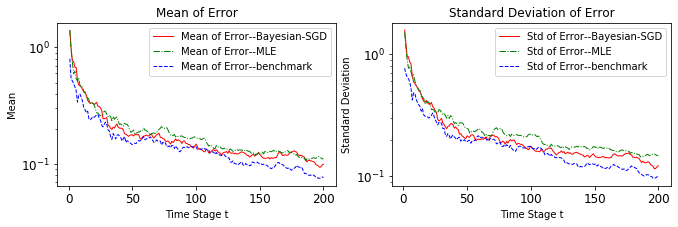}
\caption{Mean and standard deviation of $|x_t - x^{*}|$ of 100 runs of \cref{algorithm: independent} (Bayesian-SGD), MLE, and benchmark algorithm in an univariate example.}
\label{fig:decision_independent_univariate}
}
\end{figure}

We then carry out numerical experiments on a quadratic problem in a multivariate setting: $h(x,\xi)=
(x_1- 1)^2+(x_2-2)^2 + \xi (x_1+x_2)$, where $\xi$ follows an exponential distribution with mean $\theta^c$. The parameter values are as follows: $\theta^c=4$, $D=1$, $K=1$, $\Theta=\{1,2,\cdots,20\}$, $a_t=\frac{2}{t+5}$. It is easy to check $H(x,\theta^c)=(x_1+1)^2+x_2^2+4$, and the true optimal decision is taken at $x^{*}=(-1, 0)$. At each time $t$, the gradient estimator in \cref{algorithm: independent} is $\nabla_x h(x_t,\xi_t) = (2 x_1 -2 + \xi, 2 x_2 - 4 + \xi)$. We use the uniform  distribution on $\Theta$ as the prior distribution and set the initial solution $x_1=(5,5)$. We again run all three algorithms (\cref{algorithm: independent}, benchmark, MLE) for 100 times on the problem. The mean and standard deviation of the solution error $||x_t - x^{*}||_2$ over time are shown in \cref{fig:decision_independent_multivariate}, from which we can draw the same conclusion as the univariate setting.

\begin{figure}[!ht]
{
\centering
\includegraphics[width=1\textwidth]{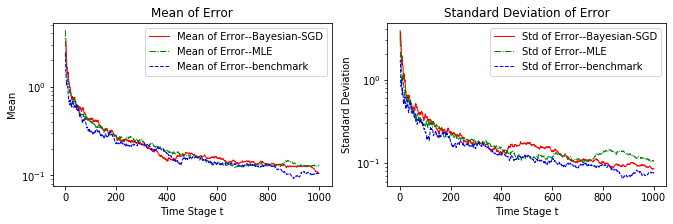}
\caption{Mean and standard deviation of $||x_t - x^{*}||_2$ of 100 runs of \cref{algorithm: independent} (Bayesian-SGD), MLE, and benchmark algorithm in a multivariate example.}
\label{fig:decision_independent_multivariate}
}
\end{figure}

\subsubsection{Decision-dependent uncertainty}
We carry out numerical experiments on a simple quadratic problem in a univariate setting: $h(x,\xi)=(x-5)^2 + 0.5 \xi x$, where $\xi \sim \mathcal{N}(x+\theta^c,\sigma^2)$.  The parameters are as follows: $\sigma=4$, $\theta^c=4$, $D=1$, $K=1$, $\Theta=\{1,2,\cdots,30\}$, $a_t=\frac{2}{t+5}$. It is easy to check $H(x,\theta) = (x-5)^2 + 0.5(x+\theta^c)x = 1.5x^2 -8x + 25$ and the true optimal decision is  $x^{*}=\frac{8}{3}$. The gradient estimator in \cref{algorithm: dependent} at each time $t$ is $\nabla_{x} h\left(x_{t}, \xi_{t}\right)+h\left(x_{t}, \xi_{t}\right) \frac{\nabla_x\widehat{f}_{t}\left(\xi_{t} ; x_{t}\right)}{ \widehat{f}_{t}\left(\xi_{t} ; x_{t}\right)}$, which can be computed as $(2x_t-10 + 0.5 \xi) + ((x_t-5)^2 + 0.5 \xi_t x_t)\frac{\sum_{\theta} \pi_t(\theta)\cdot \nabla_x f(\xi_t;x_t,\theta_t)}{\sum_{\theta} \pi_t(\theta) \cdot f(\xi_t;x_t,\theta_t)}$, where $f(\xi;x,\theta) =\frac{1}{\sqrt{2\pi}\sigma}\exp{\left(-\frac{(\xi-(x + \theta))^2}{2\sigma^2}\right)}$.

We use the uniform  distribution on $\Theta$ as the prior distribution and set the initial solution $x_1=0$. We run \cref{algorithm: dependent} and the benchmark algorithm for 100 times on the problem. Note that the MLE method in \cite{Song2019a} is not applicable for the decision-dependent case. The mean and standard deviation of the solution error $|x_t - x^{*}|$ over time are shown in \cref{fig:decision_dependent_univariate}. We further show the convergence of posterior distribution under different data batch size $D$ in \cref{fig:decision_dependent_posterior}. Note that the benchmark algorithm (without parameter uncertainty) can be viewed as \cref{algorithm: dependent} with $D = \infty$. The observations from \cref{fig:decision_dependent_univariate} and \cref{fig:decision_dependent_posterior} are summarized as follows.

\begin{itemize}
    \item With decreasing step size, the solution sequence in \cref{algorithm: dependent} converges to the true optimal solution.
    \item \cref{fig:decision_dependent_posterior} shows that as we observe more data at each time stage, the Bayesian posterior distribution converges faster to the delta function concentrated on the true parameter $\theta^c$.
    \item There is no significant difference in the convergence rate of \cref{algorithm: dependent} under different data batch sizes, even though the posterior distribution converges faster with larger data batch size. It implies that the Bayesian average of the objective function (\ref{eq: dependent}) in this example is a good estimate of the true objective function despite the inaccuracy of the posterior distribution at the beginning time stages.
\end{itemize}

\begin{figure}[!ht]
{
\centering
\includegraphics[width=1\textwidth]{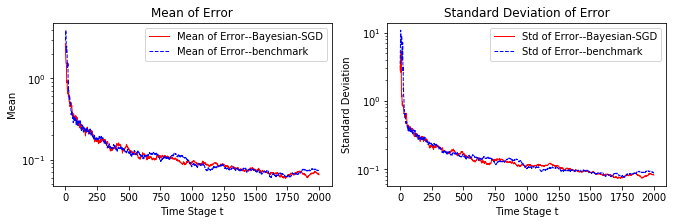}
\caption{Mean and standard deviation of $||x_t - x^{*}||_2$ of 100 runs of \cref{algorithm: dependent} (Bayesian-SGD) and  the benchmark algorithm in a univariate example.}
\label{fig:decision_dependent_univariate}
}
\end{figure}

\begin{figure}[!ht]
{
\centering
\includegraphics[width=1\textwidth]{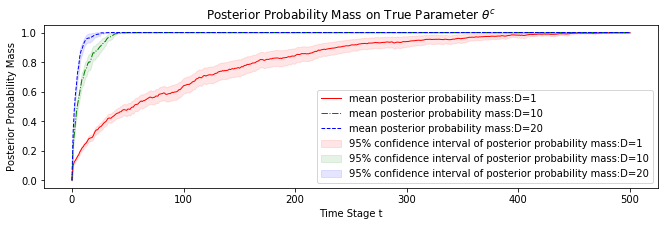}
\caption{Mean and $95\%$ confidence interval of $\pi_t(\theta^c)$ of 100 runs of \cref{algorithm: dependent} (Bayesian-SGD) under different data batch sizes.}
\label{fig:decision_dependent_posterior}
}
\end{figure}

We then carry out numerical experiments on a quadratic problem in a multivariate setting: $h(x,\xi)=(x_1-1)^2+(x_2-2)^2+\xi$, where $x=(x_1,x_2)$ and $\xi$ follows an exponential distribution with mean $(x_1-x_2)^2 + \theta^c$.  The parameters are  as follows: $\theta^c=4$, $D=1$, $K=1$, $\Theta=\{1,2,\cdots,20\}$, $a_t=\frac{2}{t+5}$. It is easy to check $H(x,\theta^c)=2 x_1^2 + 2  x_2^2 - 2 x_1  x_2 - 2 x_1 - 4  x_2 + 9$, and the true optimal decision is taken at $x^{*}=(\frac{4}{3}, \frac{5}{3})$. The gradient estimator in \cref{algorithm: dependent} can be computed as $(2 x_1-2, 2 x_2-4) + ((x_1-1)^2+(x_2-2)^2+\xi)\frac{\sum_{\theta} \pi_t(\theta)\cdot \nabla_x f(\xi_t;x_t,\theta_t)}{\sum_{\theta} \pi_t(\theta) \cdot f(\xi_t;x_t,\theta_t)}$. Recall that $f(\xi;x,\theta) = \frac{1}{(x_1-x_2)^2 + \theta} \exp{-\frac{\xi_t}{(x_1-x_2)^2 + \theta}}$. We use the uniform  distribution on $\Theta$ as the prior distribution and set the initial solution $x_1=(5,5)$. We run \cref{algorithm: dependent} and the benchmark algorithm for 100 times on the problem. The mean and standard deviation of the solution error $||x_t - x^{*}||_2$ over time are shown in \cref{fig:decision_dependent_multivariate}, from which we can draw the same conclusion as the univariate setting.

\begin{figure}[!ht]
{
\centering
\includegraphics[width=1\textwidth]{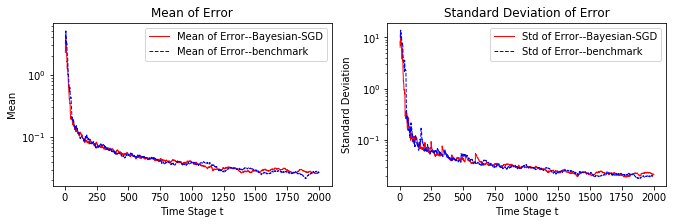}
\caption{Mean and standard deviation of $||x_t - x^{*}||_2$ of 100 runs of \cref{algorithm: dependent} (Bayesian-SGD) and the benchmark algorithm in a multivariate example.}
\label{fig:decision_dependent_multivariate}
}
\end{figure}

\subsection{Multi-item Newsvendor Problem}
We consider a multi-item newsvendor problem and its variant with decision-dependent uncertainty. In the multi-item newsvendor problem, there are $d=3$ different kinds of newspapers, and a newsboy orders $x \in \mathbb{R}^{d}_{\geq 0}$ units of newspapers to replenish the inventory at the beginning of a selling season. We assume $0 \leq x_i \leq M_i$, where $M_i$ is the inventory capacity for newspaper $i \in [d]$. During the selling season, the newsboy observes customer demands, which are observations of a random vector $\xi \in (-\infty, \infty)^{d}$ following an unknown joint distribution $F$. Negative demand implies that some customers may have bought the newspaper somewhere else and drop it off after reading. The cost of purchasing newspaper is $c$ per unit, and the selling price is $p$ per unit. At the end of the selling season the unsold newspaper has a salvage value of $s$ per unit. Note that $c, p, s$ are all 3-dimensional vectors. Also note that there is no replenishment of newspaper during the selling season. The cost function is given by $h(x,\xi) = c^{T}x - p^{T} \min (x, \max (0, \xi))  - s^{T} \max (0, x - \xi)$. Both $\min$ and $\max$ are element-wise operators. The newsboy aims to choose the amount $x$ that minimizes the expected cost, where the expectation is taken w.r.t. the distribution of $\xi$. 

\subsubsection{Decision-independent uncertainty}
We first consider the multi-item newsvendor problem with the decision-independent input uncertainty. We assume $\xi$ follows a multivariate normal distribution with mean $\theta^c_{\mu}$ and covariance matrix $\theta^c_{\Sigma}$. Note that in this problem we have 9 unknown parameters, i.e., 3 mean parameters $\theta^c_{\mu}$, 3 variance parameters and 3 correlation parameters $\theta^c_{\Sigma}:=(\theta^c_{\text{var}}, \theta^c_{\text{corr}})$. At each time $t$, the gradient estimator in \cref{algorithm: independent} is $\nabla_x h(x,\xi) = \left\{\begin{array}{ll} c-p, x \leq \xi \\ c - s, x > \xi \end{array}\right.$. The parameters are as follows: $\theta^c_{\mu}=(10, 15, 20)$, $\theta^c_{\text{var}}=(3,6,9)$, $\theta^c_{\text{corr}}=(0.1,0.3,0.5)$, thus the true covariance matrix is $((3, 0.42, 1.56), (0.42, 6, 3.67), (1.56, 3.67, 9))$; parameter space $\Theta_{\mu}=\{5, 10, 15, 20, 25\}^{3}$, 
\newline
$\Theta_{\text{var}}=\{1,3,6,9,12\}^{3}$, $\Theta_{\text{corr}}=\{0.1,0.2,0.3,0.4,0.5\}^{3}$; $D=2$, $K=1$, $M=(100,100,100)$, $c=(2,4,6)$, $p=(4,6,8)$, $s=(1,2,3)$, $a_t=\frac{2}{t+5}$. We denote by $x^{*}$ the optimal decision under the true parameters. We use the uniform distribution on $\Theta$ as the prior distribution and set the initial solution $x_1=(15,15,15)$. We run all three algorithms (\cref{algorithm: independent}, benchmark, MLE) for 100 times on the problem. The mean and standard deviation of the solution error $||x_t - x^{*}||_2$ over time are shown in \cref{fig:decision_independent_newsvendor}. We have similar observations as the synthetic quadratic problem.

\begin{figure}[!ht]
{
\centering
\includegraphics[width=1\textwidth]{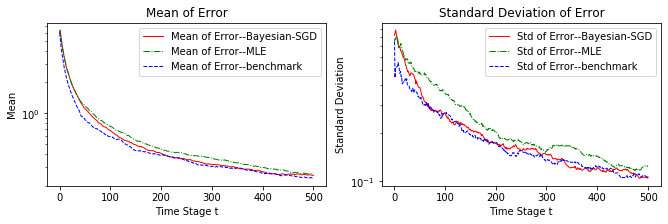}
\caption{Mean and standard deviation of $||x_t - x^{*}||_2$ of 100 runs of \cref{algorithm: independent} (Bayesian-SGD), MLE, and the benchmark algorithm in the multi-item newsvendor problem with decision-independent data.}
\label{fig:decision_independent_newsvendor}
}
\end{figure}

% Compared to the quadratic problem with decision-independent uncertainty, all three algorithms converge much more slowly, possibly due to the large variance of the randomness.

\subsubsection{Decision-dependent uncertainty}
We then consider the multi-item newsvendor problem with the decision-dependent input uncertainty, where the customer demand depends on the order amount $x$ of the inventory. We follow the setting in \cite{balakrishnan2004stack}, in which high inventory stimulates demand. We assume the demand $\xi$ follows a multivariate normal distribution with mean $\theta^c_{\mu} + \alpha x^{\beta}$ and covariance matrix $\theta^c_{\Sigma}$, where $\alpha > 0, 0 < \beta <1$ are vectors and $(\cdot)^{\beta}$ is element-wise operator. Note that the mean function admits diminishing marginal utility, which says that the marginal increase in the mean demand diminishes as the inventory level increases. The gradient estimator in \cref{algorithm: dependent} at each time stage $t$ is given by 
$$\nabla_{x} h\left(x_{t}, \xi_{t}\right)+h\left(x_{t}, \xi_{t}\right) \frac{\sum_{\theta} \pi_t(\theta)\cdot \nabla_x f(\xi_t;x_t,\theta_t)}{\sum_{\theta} \pi_t(\theta) \cdot f(\xi_t;x_t,\theta_t)}.$$
$f(\xi;x,\theta)=\frac{\exp \left(-\frac{1}{2}(\xi-(\theta_{\mu}+\alpha x^{\beta}))^{T} \theta_{\Sigma}^{-1}(\xi-(\theta_{\mu}+\alpha x^{\beta}))\right)}{\sqrt{(2 \pi)^{d}|\theta_{\Sigma}|}}$, $\nabla_x f(\xi;x,\theta)=f(\xi;x,\theta) \theta_{\Sigma}^{-1}(\xi-(\theta_{\mu}+\alpha x^{\beta}))\alpha \beta x^{\beta-1}$. The parameters are as follows: $\theta^c_{\mu}=(10, 15, 20)$, $\theta^c_{\text{var}}=(3,6,9)$, $\theta^c_{\text{corr}}=(0.1,0.3,0.5)$, the true covariance matrix is $((3, 0.42, 1.56), (0.42, 6, 3.67), (1.56, 3.67, 9))$. $\Theta_{\mu}=\{5, 10, 15, 20, 25\}^{3}$, $\Theta_{\text{var}}=\{1,3,6,9,12\}^{3}$, $\Theta_{\text{corr}}=\{0.1,0.2,0.3,0.4,0.5\}^{3}$. $D=2$, $K=1$, $M=(100,100,100)$, $c=(2,4,6)$, $p=(4,6,8)$, $s=(1,2,3)$, $\alpha=1, \beta=0.5$, $a_t=\frac{2}{t+5}$. We denote by $x^{*}$ the optimal decision under the true parameters. We use the uniform distribution on $\Theta$ as the prior distribution and set the initial solution $x_1=(15,15,15)$. We run \cref{algorithm: dependent} and the benchmark algorithm for 100 times on the problem. The mean and standard deviation of the solution error $||x_t - x^{*}||_2$ over time are shown in \cref{fig:decision_dependent_newsvendor}. We have similar observations as the synthetic quadratic problem.

\begin{figure}[!ht]
{
\centering
\includegraphics[width=1\textwidth]{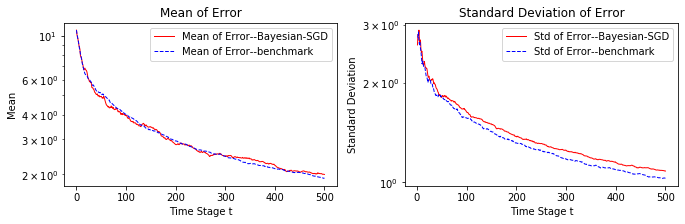}
\caption{Mean and standard deviation of $||x_t - x^{*}||_2$ of 100 runs of \cref{algorithm: dependent} (Bayesian-SGD) and the benchmark algorithm in the multi-item newsvendor problem with decision-dependent data.}
\label{fig:decision_dependent_newsvendor}
}
\end{figure}

As a final note, the good performance of our proposed algorithms on the multi-dimensio-nal newsvendor problem shows promise of the applicability of our proposed approaches to large-scale problems. However, it should be noted that most of the computational time is devoted to the posterior updating, especially for the high-dimensional problem where there is no conjugate prior. It would be interesting to adapt our algorithms to such a high-dimensional setup, where we could leverage the recent theoretical results of Bayesian procedures in high-dimension (cf. \cite{durmus2016high, chopin2015some}).

\section{Conclusions}
\label{sec: conclusions}
In this paper, we propose a Bayesian-SGD approach to stochastic optimization with streaming input data, and present two algorithms for decision-independent and decision-dependent uncertainty respectively. We show the asymptotic convergence of both algorithms, and derive the convergence rate in the decision-independent case based on the non-asymptotic analysis of the Bayesian estimate. Our consistency result of Bayesian posterior distribution with decision-dependent input data could be of independent interest to Bayes estimation. Note that our approach can be viewed as an online extension of the BRO framework \cite{Zhou2015,Wu2018}, and it would be interesting to adapt our approach to other risk functionals (such as Value-at-Risk and Conditional Value-at-Risk) with respect to the unknown distributional parameter.

\appendix
\section{Proof of \cref{lemma: independent_concentration}}
\label{appendix: A}
\begin{proof}
Define the Hellinger distance between $\theta_1$ and $\theta_2$ as
\begin{align*}
    d(\theta_1,\theta_2) =\sqrt{\frac{1}{2}\int_\mathcal{Y} (\sqrt{f(y;\theta_1)}-\sqrt{f(y;\theta_2)})^2}.
\end{align*}
One can easily verify that there exist a constant $A$ such that $\|\theta_1-\theta_2\|\leq Ad(\theta_1,\theta_2)$, where $||\cdot||$ is the Euclidean norm. Let $B_k^t = B(\theta^c, k/\sqrt{Dt})$ be a ball centered at $\theta^c$ with radius $k/\sqrt{Dt}$ under distance $d.$ Since $\Theta$ is finite, we can directly apply Proposition 1 in \cite{birge2015non}. Then for $t\leq T, \epsilon,\delta\in(0,1)$ with probability at least $1-\frac{6\delta}{\pi^2 t^2}$ with respect to $\mathbb{P}_{\theta^c}^{t}$, we have
\begin{align*}
	\pi_t(B_{k(t)}^t) \geq 1-\epsilon,
\end{align*}
where 
\begin{align*}
    k(t) =\inf\left\{j\geq 1\Big| \sum_{i\geq j}|\Theta|e^{-i^2}\leq \frac{6\delta}{\pi^2 t^2}\sqrt{\epsilon \pi_0(\theta^c)}\right\}.
\end{align*}
Note that $\sum_{i\geq j}e^{-i^2}\leq \frac{e}{e-1} e^{-j^2}$, we can set $k(t)$ to be the solution of next equation. 
\begin{align*}
    \frac{e}{e-1}|\Theta| e^{-k(t)^2} =\frac{6\delta}{\pi^2 t^2}\sqrt{\epsilon \pi_0(\theta^c)}.
\end{align*}
By simple calculation, we have $k(t) = \sqrt{\log \frac{e|\Theta|\pi^2 t^2}{6\delta(e-1)\sqrt{\epsilon\pi_0(\theta^c)}}}$. Now we are ready to bound the bias in the gradient estimator.
\begin{align*}
	& \|\mathbb{E}_{\pi_t}\nabla_x H(x, \theta)-\mathbb{E}_{\pi_t}\nabla_x H(x, \theta^c)\|_2^2 \\
	& = \left\|\int (\nabla_x H(x, \theta)- \nabla_x H(x, \theta^c))\pi_t(\theta) d\theta\right\|_2^2\\
	&\leq \int  \left\|(\nabla_x H(x, \theta)- \nabla_x H(x, \theta^c))\right\|_2^2\pi_t(\theta) d\theta\\
	&\leq  \int L_{H}^2||\theta-\theta^c||_2^2 \pi_t(\theta) d\theta\\
	&=  \int_{B_{k(t)}^t} L_{H}^2||\theta-\theta^c||_2^2 \pi_t(\theta) d\theta + \int_{(B_{k(t)}^t)^c} L_{H}^2||\theta-\theta^c||_2^2 \pi_t(\theta) d\theta\\
	&\leq A^2 L_{H}^2 \frac{k(t)^2}{Dt}  \int_{B_{k(t)}^t} \pi_t(\theta) d\theta + L_{H}^2 \max_{\theta\in\Theta}\|\theta-\theta^c\|^2_2\int_{(B_{k(t)}^t)^c}  \pi_t(\theta) d\theta\\
	&\leq  A^2 L_{H}^2 \frac{k(t)^2}{Dt} +  L_{H}^2 \max_{\theta\in\Theta}\|\theta-\theta^c\|^2_2\epsilon.
\end{align*}
Recall that $D$ is the data batch size. Take $\epsilon =\frac{1}{Dt}$, note that $k(t) =\sqrt{\log \frac{e|\Theta|\pi^2 t^2\sqrt{Dt}}{6\delta(e-1)\sqrt{\pi_0(\theta^c)}}}$. We further have
\begin{align*}
	\|\mathbb{E}_{\pi_t}\nabla_x H(x, \theta)-\mathbb{E}_{\pi_t}\nabla_x H(x, \theta^c)\|_2^2
	&\leq  A^2 L_{H}^2 \frac{k(t)^2}{Dt} +  L_{H}^2 \max_{\theta\in\Theta}\|\theta-\theta^c\|^2_2\epsilon\\
	&\leq 2A^2 L_{H}^2\max_{\theta\in\Theta}\|\theta-\theta^c\|^2_2\frac{\log \frac{e|\Theta|\pi^2 t^2\sqrt{Dt}}{6\delta(e-1)\sqrt{\pi_0(\theta^c)} }}{Dt}\\
	&= O(\frac{\log Dt + \log \frac{1}{\delta}}{Dt}).
\end{align*}

Let $\mathcal{E}_t$ denote the event that the above inequality holds, and $\mathcal{E}_t^c$ denote the complement event. Then we have $\mathbb{P}(\mathcal{E}_t^c) \leq \frac{6\delta}{\pi^2 t^2}$. Therefore,
\begin{align*}
    \mathbb{P}(\cap_{t=1}^\infty \mathcal{E}_t) & = 1- \mathbb{P}(\bigcup_{t=1}^{\infty} \mathcal{E}^c_t) \\
    & \geq 1 - \sum_{t=1}^{\infty} \mathbb{P}(\mathcal{E}_t^c) \quad (\text{union bound}) \\
    & \geq 1 - \sum_{t=1}^{\infty} \frac{6\delta}{\pi^2 t^2} \\
    & = 1 - \delta.
\end{align*}
\end{proof}

\section{Proof of \cref{thm: independent_rate}}
\label{appendix: B}
\begin{proof}
By the update \cref{eq: SGD_independent_2}, we know that for any $t\leq T$,
\begin{align*}
    x_{t+1} &= x_{t} - a_{t} \nabla_x H(x_{t},\theta^c) -  a_{t}[\mathbb{E}_{\pi_t}\nabla_x H(x_{t}, \theta)-\nabla_x H(x_{t}, \theta^c)]	\\
    &~~~~~~~~- a_{t}[\nabla_x h(x_{t}, \xi_{t})-\mathbb{E}_{\pi_t}\nabla_x H(x_{t}, \theta)]\\
    & =  x_{t} - a_{t}	\nabla_x H(x_{t},\theta^c) -  a_{t} B_{t} - a_{t} N_{t},
\end{align*}
where $B_{t}$ is the bias and $N_{t}$ is the noise. By \cref{lemma: independent_concentration}, we know $\mathbb{E}[\|B_{t}\|_2^2]\leq  C_1 \frac{\log Dt + \log \frac{1}{\delta}}{Dt}.$ From \cref{ass: independent_rate} we have $\mathbb{E}[\|N_{t}\|_2^2]\leq \sigma^2.$ By the proof of Lemma 2 in \cite{ajalloeian2020analysis}, we know that 
\begin{align*}
    \mathbb{E}[H(x_{t+1},\theta^c)] - H(x_{t},\theta^c)\leq -\frac{a_{t}}{2}\|\nabla_x H(x_{t},\theta^c) \|_2^2 +\frac{a_{t}}{2}C_1 \frac{\log Dt + \log \frac{1}{\delta}}{Dt} +\frac{a_{t}^2}{2}L_{h} \sigma^2,
\end{align*} 
Rearranging the terms in the inequality above, summing over $t$ from 1 to $T$, and noting that $H(x_{t}, \theta^c) \leq \min_{x \in \mathcal{X}} H(x, \theta^c), \forall t$, we have 
{\small
\begin{align*}
    \sum_{t=1}^{T} a_t \mathbb{E}[\|\nabla_x H(x_{t},\theta^c) \|_2^2]\leq 2
    (H(x_1,\theta^c) - \min_{x \in \mathcal{X}} H(x,\theta^c))
    + C_1 \sum_{t=1}^{T} a_t \frac{\log Dt + \log \frac{1}{\delta}}{Dt} + L_{h}\sigma^2 \sum_{t=1}^{T} a_t^2,
\end{align*}
}
Dividing both sides of the above inequality by $\sum_{t=1}^{T} a_t$, and noting that 
\begin{align*}
    \mathbb{E}[\|\nabla_x H(z_T,\theta^c) \|_2^2] = \frac{1}{\sum_{t=1}^{T} a_t} \sum_{t=1}^{T} a_t \mathbb{E}[\|\nabla_x H(x_{t},\theta^c) \|_2^2],
\end{align*}
we have
{\small
\begin{align*}
    \mathbb{E}[\|\nabla_x H(z_T,\theta^c) \|_2^2] \leq \frac{1}{\sum_{t=1}^{T} a_t}\left[2 (H(x_1,\theta^c) - \min_{x \in \mathcal{X}} H(x,\theta^c)) +C_1 \sum_{t=1}^{T} a_t \frac{\log Dt + \log \frac{1}{\delta}}{Dt} + L_{h} \sigma^2\sum_{t=1}^{T}a_t^2.
    \right]
\end{align*}
}
(i) $a_t=\frac{a}{\sqrt{T}}$, $\forall t \leq T$, for some constant $a<\frac{\sqrt{T}}{L_{h}}$. Note that $\sum_{t=1}^{T} \frac{1}{t} \leq \log T+1$ and $\sum_{t=1}^{T} \frac{\log t}{t} \leq \log (\log T+1)$.
Then
{\small
\begin{align*}
    & \mathbb{E}[\|\nabla_x H(z_T,\theta^c) \|_2^2]\\
    &\leq \frac{2(H(x_1,\theta^c)- \min_{x}H(x,\theta^c))}{a\sqrt{T}} + \frac{C_1 (\log D - \log \delta) (\log T + 1)}{L_{h} D T} + \frac{C_1 \log T (\log T + 1)}{L_{h} D T} \\ 
    & = \frac{2(H(x_1,\theta^c)- \min_{x}H(x,\theta^c))}{a\sqrt{T}} + \frac{C_1 (\log D - \log \delta)}{L_{h} D T} + \frac{C_1 (\log D - \log \delta) \log T}{L_{h} D T} + \frac{C_1 \log^2 T}{L_{h} D T} + \frac{L_{h} a\sigma^2}{\sqrt{T}}
\end{align*}
}
(ii) $a_t = \frac{a}{t}$, $\forall t \leq T$, for some constant $a < \frac{1}{L_{h}}$. Let $M_T=\sum_{t=1}^{T}\frac{1}{t}$. Note that 
\begin{align*}
    \sum_{t=1}^{T}\frac{\log t}{t^2} < \sum_{t=1}^{\infty}\frac{\log t}{t^2} = \frac{\pi^2}{6}(12 \ln A - \gamma - \ln 2\pi) < 1,
\end{align*}
where $A \approx 1.28$ is the Glaisher-Kinkelin constant and $\gamma \approx 0.58$ is the Euler-Mascheroni constant. Then we have
\begin{align*}
	& \mathbb{E}[\|\nabla_x H(z_T,\theta^c) \|_2^2]\\
	&\leq \frac{2(H(x_1,\theta^c)- \min_{x\in\mathcal{X}}H(x,\theta^c))}{a M_T} +\frac{C_1}{M_T}\sum_{t=1}^T \frac{\log Dt + \log \frac{1}{\delta}}{Dt^2} +\sum_{t=1}^T \frac{L_{h} a \sigma^2}{M_T t^2}\\
	& \leq \left[\frac{2(H(x_1,\theta^c)- \min_{x\in\mathcal{X}}H(x,\theta^c))}{a} +\frac{6C_1 + \pi^2C_1 (\log D - \log \delta)}{6D}+\frac{\pi^2 L_{h} a \sigma^2}{6}\right]\frac{1}{\log T}.
\end{align*}
(iii) $a_t = \frac{a}{\sqrt{t}}$, $\forall t \leq T$, for some constant $a < \frac{1}{L_{h}}$. Let $Q_t=\sum_{t=1}^{T}\frac{1}{\sqrt{t}}$. Note that $\sum_{t=1}^{\infty}\frac{1}{t\sqrt{t}}=\zeta(1.5)\approx2.61<3$, $\sum_{t=1}^{\infty} \frac{\log t}{t\sqrt{t}} < 4$, $\sum_{t=1}^{T}\frac{1}{\sqrt{t}} \geq \sqrt{T}$, where $\zeta(\cdot)$ is the Riemann’s zeta function. Then we have
{\small
\begin{align*}
	& \mathbb{E}[\|\nabla_x H(z_T,\theta^c) \|_2^2]\\
	&\leq \frac{2(H(x_1,\theta^c)- \min_{x}H(x,\theta^c))}{a Q_T} +\frac{C_1(\log D - \log \delta)}{D Q_T}\sum_{t=1}^T \frac{1}{t\sqrt{t}} + \frac{C_1}{D Q_T}\sum_{t=1}^{T}\frac{\log t}{t \sqrt{t}} +  \frac{L_{h} a \sigma^2}{Q_T} \sum_{t=1}^T \frac{1}{t}\\
	&\leq \left[\frac{2(H(x_1,\theta^c)- \min_{x}H(x,\theta^c))}{a\sqrt{T}} + \frac{3C_1(\log D - \log \delta) + 4C_1}{D\sqrt{T}} + \frac{L_{h} a \sigma^2}{\sqrt{T}}\right] + \frac{L_{h} a \sigma^2 \log T}{\sqrt{T}}.
\end{align*}
}
\end{proof}

\section{Proof of \cref{lemma: consistency}}
\label{appendix: D}
\begin{proof}
Define $w_t = -\log \pi_t(\theta^c).$ One can easily verify that $w_t\geq 0$. Then we have 
\begin{align*}
	\mathbb{E}[w_{t+1}] &= \mathbb{E}\left[\mathbb{E}[w_{t+1}|\mathcal{F}_t,x_{t+1}]\right]\\
	&=  \mathbb{E}\left[\mathbb{E}\left[-\log \frac{\pi_{t}(\theta^c) f(y_{t+1};x_{t+1},\theta^c)}{\sum_\theta \pi_{t}(\theta) f(y_{t+1};x_{t+1},\theta)}|\mathcal{F}_t,x_{t+1}\right]\right]\\
	&= \mathbb{E}\left[-\log \pi_{t}(\theta^c)-\mathbb{E}\left[\log \frac{f(y_{t+1};x_{t+1},\theta^c)}{\sum_\theta \pi_{t}(\theta) f(y_{t+1};x_{t+1},\theta)}|\mathcal{F}_t,x_{t+1}\right]\right]\\
	&= \mathbb{E}[w_t]-\mathbb{E}[D_{KL}(f^*(\cdot;x_{t+1})||\hat f_t(\cdot;x_{t+1}))].
\end{align*}
This implies that $\mathbb{E}[d_t]=\mathbb{E}[w_t]-\mathbb{E}[w_{t+1}].$ For any $T>0$, we have
$$\sum_{t=0}^T \mathbb{E}[d_t] =\sum_{t=0}^T\mathbb{E}[w_t]-\mathbb{E}[w_{t+1}] =w_0-\mathbb{E}[w_{T+1}]\leq w_0<\infty.$$
Then we have $\sum_{t=0}^\infty \mathbb{E}[d_t]\leq w_0.$ $\forall \epsilon>0,$ we have
$$\sum_{t=0}^\infty \mathbb{P}(d_t\geq \epsilon)\leq \frac{1}{\epsilon}\sum_{t=0}^\infty\mathbb{E}[d_t]<\infty.$$
By Borel-Cantelli Lemma, we know that $\mathbb{P}(d_t\geq\epsilon, i.o.)=0,$ where $i.o.$ stands for infinitely often. It then implies
$\lim_{t\rightarrow\infty} d_t =0, ~\text{w.p.1} (\mathbb{P}^{\infty}_{\theta^c})$. Moreover, since $d_t\geq 0,$ by Tonelli's Theorem, we have
$$\mathbb{E}\left[\sum_{t=0}^\infty d_t\right] =\sum_{t=0}^\infty \mathbb{E}[d_t]\leq w_0.$$
Since $\sum_{t=0}^\infty d_t$ has bounded expectation, it must be finite w.p.1 $(\mathbb{P}^{\infty}_{\theta^c})$.
\end{proof}

\section{Proof of \cref{lemma: pi_consistency}}
\label{appendix: E}
\begin{proof}
Without loss of generality, we assume that $\theta^c=\theta_1.$ Recall that  $f^*(\xi;x_{t+1}) =f^*(\xi;x_{t+1},\theta_1)$ and $\hat f_t(\xi;x_{t+1})=\sum_i \pi_t(\theta_i)f(\xi;x_{t+1},\theta_i).$ Then we have 
\begin{equation}\label{eqn:measurement distance}
    f^*(\xi;x_{t+1})-\hat f_t(\xi;x_{t+1}) = (1-\pi_t(\theta_1))f(\xi;x_{t+1},\theta_1)-\sum_{i>1} \pi_t(\theta_i)f(\xi;x_{t+1},\theta_i).
\end{equation}
Note that for any $t > 0$, $(\pi_{t}(\theta_1), \pi_{t}(\theta_2), \cdots)$ is infinitely dimensional bounded vector with all components in the interval $[0,1]$ and sum up to 1 (normalized), we can take a subsequence $\{\pi_{t_k}\}$ such that for each component $j$, $\pi_{t_k}(\theta_j)$ converges to a limit which is denoted by $\pi_{\infty}(\theta_j)$, which is also known as weak convergence (of a deterministic sequence). Next, we will show that $\pi_{\infty}(\theta)$ is a normalized vector. For any $j \in \mathbb{N}$, $\lim_{t_k \to \infty} \pi_{t_k}(\theta_j) = \pi_{\infty}(\theta_j)$, which is equivalent to
\begin{equation*}
    \forall \epsilon_j > 0, \exists N \in \mathbb{N}, s.t. \forall n \geq N, |\pi_{\infty}(\theta_j) - \pi_{n}(\theta_j)| \leq \epsilon.
\end{equation*}
Therefore, we have
\begin{equation}\label{eqn: posterior_sum}
    - \epsilon_j < \pi_{\infty}(\theta_j) - \pi_n(\theta_j) < \epsilon_j, j=1,2,\cdots
\end{equation}
According to the Bayesian update rule, we know $\sum_{j=1}^{\infty}\pi_{n}(\theta_j) = 1$. It then follows that $\forall \epsilon > 0$, take $\epsilon_j = \frac{\epsilon}{2^{j}}$ and sum over \eqref{eqn: posterior_sum} for all $j \in \mathbb{N}$, we get
\begin{equation*}
    -(\frac{\epsilon}{2^{1}} + \frac{\epsilon}{2^{2}} + \cdots) <  \sum_{j=1}^{\infty} \pi_{\infty}(\theta_j) - 1 < (\frac{\epsilon}{2^{1}} + \frac{\epsilon}{2^{2}} + \cdots),
\end{equation*}
which indicates $\forall \epsilon > 0$, $|\sum_{j=1}^{\infty} \pi_{\infty}(\theta_j)-1| < \epsilon$, and it implies that $\sum_{j=1}^{\infty} \pi_{\infty}(\theta_j) = 1$. So the limit is also a valid probability simplex. Since every weakly convergent sequence in $L^{1}$ is strongly convergent (cf. Chapter 2 in \cite{pedersen2012analysis}), we can take any convergent subsequence of $\{\pi_{t_k}\}$ with limit $(p_1^{*}, p_2^{*}, \cdots)$. Since $\mathcal{X}$ is also bounded, from this subsequence, we could take a further subsequence $\{\pi_{\tau_k}\}$ with time stage $\tau_1,\tau_2,\cdots$, such that $\{x_{\tau_k}\}$ converges to some $x^{\prime}$. Then take limit over \eqref{eqn:measurement distance} along $\tau_1, \tau_2, \cdots$, we have 
$$f^*(\xi;x_{{\tau_k}})-\hat f_{\tau_k}(\xi;x_{{\tau_k}}) \rightarrow (1-(p_1^*))f(\xi;x',\theta_1)-\sum_{i>1} p^*_i f(\xi;x',\theta_i).$$
Moreover, since K-L divergence dominates total variation distance between two distributions, we have
\begin{align}\label{eq:pi_consistency}
	\int_\Xi \left|f^*(\xi;x_{t+1})-\hat f_t(\xi;x_{t+1})\right|d\xi\leq d_t.
\end{align}
From \cref{eq:pi_consistency} and \cref{lemma: consistency}, we know that $\int_\Xi \left|f^*(\xi;x_{t+1})-\hat f_t(\xi;x_{t+1})\right|d\xi\rightarrow 0$ w.p.1 ($\mathbb{P}_{\theta^c}^{\infty}$). By DCT, we have 
$$\int_\Xi\left|(1-(p_1^*))f(\xi;x^{\prime},\theta_1)-\sum_{i>1} p^*_i f(\xi;x^{\prime},\theta_i)\right|d\xi = 0,$$ which implies: $$(1-(p_1^*))f(\xi;x^{\prime},\theta_1)-\sum_{i>1} p^*_i f(\xi;x^{\prime},\theta_i)=0, \forall \xi.$$
By linear independence, we know $p^*_1= 1,p^*_2=p^*_3=...=0.$ Since every convergent subsequence of $\{(\pi_t(\theta_1),\pi_t(\theta_2),\cdots)\}_t$ has the same limit, we have $\pi_t \Rightarrow \delta_{\theta^c}$ w.p.1 ($\mathbb{P}_{\theta^c}^{\infty}$).
\end{proof}

\section{Derivation of unbiased estimator in decision-dependent case}
\label{appendix: C}
\begin{align*}
    \nabla_x \mathbb{E}_{\pi_t}[H(x,\theta)] & = \mathbb{E}_{\pi_t}\left[\nabla_x \mathbb{E}_{f(\cdot;x,\theta)}[h(x,\xi)]\right] \\
    & = \mathbb{E}_{\pi_t} \left[\int_{\Xi} \nabla_x h(x,\xi)f(\xi;x,\theta)d\xi\right] +  \mathbb{E}_{\pi_t} \left[\int_{\Xi}h(x,\xi)\nabla_xf(\xi;x,\theta)d\xi\right] \\ 
    & = \mathbb{E}_{\pi_t} \left[\mathbb{E}_{f(\cdot;x,\theta)}[\nabla_x h(x,\xi)]\right] + \int_{\Theta}\left( \int_{\Xi} h(x,\xi) \nabla_x f(\xi;x,\theta)d\xi\right) \pi_t(\theta) d\theta\\
    & = \mathbb{E}_{\pi_t} \left[\mathbb{E}_{f(\cdot;x,\theta)}[\nabla_x h(x,\xi)]\right] + \int_{\Xi} h(x,\xi) \left( \int_{\Theta} \pi_t(\theta) \nabla_x f(\xi;x,\theta)d\theta \right) d\xi\\
    & = \mathbb{E}_{\pi_t} \left[\mathbb{E}_{f(\cdot;x,\theta)}[\nabla_x h(x,\xi)]\right] + \int_{\Xi} h(x,\xi) \nabla_x \hat{f}(\xi;x) d\xi \\
    & = \mathbb{E}_{\pi_t} \left[\mathbb{E}_{f(\cdot;x,\theta)}[\nabla_x h(x,\xi)]\right] + \int_{\Xi} h(x,\xi)\frac{\nabla_x \hat{f}(\cdot;x)}{\hat{f}(\cdot;x)} \hat{f}(\cdot;x)d\xi \\
    & = \mathbb{E}_{\pi_t} \left[\mathbb{E}_{f(\cdot;x,\theta)}[\nabla_x h(x,\xi)]\right] + \mathbb{E}_{\hat{f}(\cdot;x)}\left[h(x,\xi)\frac{\nabla_x \hat{f}(\cdot;x)}{\hat{f}(\cdot;x)}\right] \\
    & = \mathbb{E}_{\pi_t} \left[\mathbb{E}_{f(\cdot;x,\theta)}[\nabla_x h(x,\xi) + h(x,\xi) \frac{\nabla_x \hat{f}(\cdot;x)}{\hat{f}(\cdot;x)}]\right].
\end{align*}
From \cref{ass: independent_DCT} and \cref{ass: dependent_DCT}, we know that both the objective function $h(x,\xi)$ and the density function $f(\xi;x,\theta)$ are $C^{1}$-smooth. The Lipschitz continuous gradient implies both $h(x,\xi)$ and $f(\xi;x,\theta)$ are integrable functions; $\nabla_x h(x,\xi)$ and $\nabla_x f(x,\xi)$ are dominated by some integrable functions. Using the chain rule, we have $\nabla_x h(x,\xi) f(\xi;x,\theta)=\nabla_x h(x,\xi) \cdot f(\xi;x,\theta)+h(x,\xi) \cdot \nabla_x f(\xi;x,\theta)$, and thus $\nabla_x h(x,\xi) f(\xi;x,\theta)$ is dominated by some integrable function. The second equality holds as the interchange between expectation and differentiation is justified by DCT. The first equality is also justified by DCT in a similar manner. Also note that since $h(x,\xi)\nabla_x f(\xi;x,\theta)$ is dominated by some integrable function, it is also absolutely integrable, hence the fourth equality is justified by Fubini-Tonelli theorem. 

\section{Proof of \cref{proposition: bias vanish}}
\label{appendix: F}
\begin{proof}
We bound $|\beta_{t,1}|$ as follows.
\begin{align*}
	& |\beta_{t,1}|  = | \mathbb{E}_{\hat{f_t}(\cdot;x_t)}\nabla_x h(x_t,\xi) - \mathbb{E}_{f^*(\cdot;x_t)}\nabla_x h(x_t,\xi)|\\
	&\leq \max_{x, \xi}|\nabla_x h(x,\xi)| \int_\Xi \left|f^*(\xi;x_{t})-\hat f_t(\xi;x_{t})\right| d \xi \\
	&\leq L_{h}^{\prime} \int_\Xi |f^*(\xi;x_{t})-f^*(\xi;x_{t+1})| + |f^*(\xi;x_{t+1})-\hat f_t(\xi;x_{t+1})| + |\hat f_t(\xi;x_{t})-\hat f_t(\xi;x_{t+1})| d\xi.
\end{align*}
From \cref{ass: independent_DCT} we know $h(x,\xi)$ is continuously differentiable, which implies it has bounded gradient, such that $|\nabla_x h(x,\xi)| \leq L_h^{\prime}$ for some $L_h^{\prime} > 0$. From \cref{ass: dependent_DCT}, we know $f(\xi;x,\theta)$ is continuously differentiable, which implies it has bounded gradient, such that $|\nabla_x f(\xi;x,\theta)| \leq L_f^{\prime}$ for some $L_f^{\prime} > 0$. Therefore, for every $\xi \in \Xi$,
\begin{align}\label{appendix: proposition proof 1}
    |f^*(\xi;x_{t})-f^*(\xi;x_{t+1})| \leq L_f' |x_t-x_{t+1}| \leq L_f' D_f a_t,
\end{align}
\begin{align}\label{appendix: proposition proof 2}
    |\hat f_t(\xi;x_{t})-\hat f_t(\xi;x_{t+1})| \leq L_f' D_f a_t,
\end{align}
for some $D_f > 0$. Let $q^{*}_t (\xi)=f^{*}(\xi;x_t)-f^{*}(\xi;x_{t+1})$. Since $\lim_{t \to \infty} a_t=0$, and by \eqref{appendix: proposition proof 1}, we have $\lim_{t \to \infty} |q^{*}_t(\xi)|=0$ for every $\xi \in \Xi$. By absolute value theorem, we have $\lim_{t \to \infty} q^{*}_t(\xi)=0$ for every $\xi \in \Xi$, that is, $q^{*}_t$ converges pointwise to $0$. By dominated convergence theorem, we have 
\begin{align}\label{appendix: proposition proof 3}
    \lim_{t \to \infty} \int_\Xi |f^*(\xi;x_{t})-f^*(\xi;x_{t+1})|=0.
\end{align}
Similarly, let $\hat{q}_t (\xi)=\hat{f}_t(\xi;x_t)-\hat{f}_t(\xi;x_{t+1})$. From \cref{ass: dependent}, we have $\lim_{t \to \infty} a_t=0$, and by \eqref{appendix: proposition proof 2}, we have $\lim_{t \to \infty} |\hat{q}_t(\xi)|=0$ for every $\xi \in \Xi$. By absolute value theorem, we have $\lim_{t \to \infty} \hat{q}_t(\xi)=0$ for every $\xi \in \Xi$, that is, $\hat{q}_t$ converges pointwise to $0$. By dominated convergence theorem, we have 
\begin{align}\label{appendix: proposition proof 4}
    \lim_{t \to \infty} \int_\Xi |\hat{f}_t(\xi;x_{t})-\hat{f}_t(\xi;x_{t+1})|=0.
\end{align}
Moreover, since K-L divergence dominates total variation distance between two distributions, we have
\begin{align}
	\int_\Xi |f^*(\xi;x_{t+1})-\hat f_t(\xi;x_{t+1})| d\xi\leq d_t.
\label{appendix: proposition proof 5}
\end{align}
From \cref{lemma: consistency}, we have $\lim_{t \to \infty} d_t=0$ w.p.1 ($\mathbb{P}_{\theta^c}^{\infty}$). Combining \eqref{appendix: proposition proof 3},  \eqref{appendix: proposition proof 4}, and \eqref{appendix: proposition proof 5} together, we know that $
\lim_{t \to \infty}	|\beta_{t,1}| =0$ w.p.1 $(\mathbb{P}_{\theta^c}^{\infty})$. 
\end{proof}

\section{Proof of \cref{proposition: consistency}}
\label{appendix: G}
\begin{proof}
We bound $|\beta_{t,2}|$ as follows. From \cref{ass: independent_DCT} we know $h(x,\xi)$ is continuously differentiable, which implies it is an integrable function of $\xi$ for every $x \in \mathcal{X}$. Thus, $\int_{\Xi} h(x,\xi) d\xi = U_h$ for some $-\infty < U_h < \infty$. From \cref{ass: dependent_DCT} we know $f(\xi;x,\theta)$ is continuously differentiable, which implies it has bounded gradient, such that $|\nabla_x f(\xi;x,\theta)| \leq L_f'$ for some $L_f' > 0$. 
\begin{align*}
	|\beta_{t,2}|&= \left| \int_\Xi h(x_t,\xi)\nabla_x\widehat{f}_{t}\left(\xi; x_{t}\right) d \xi - \int_\Xi h(x_t,\xi)\nabla_x f^*(\xi;x_t) d \xi\right|\\
	&=\left| \int_\Xi h(x_t,\xi)\left(\nabla_x\widehat{f}_{t}\left(\xi; x_{t}\right) -\nabla_x f^*(\xi;x_t)\right )d \xi\right|\\
	&= \left| \int_\Xi h(x_t,\xi)\left(\sum_{\theta\in \Theta}(\pi_t(\theta)-\delta_{\theta^c}(\theta))\nabla_x f(\xi;x_t, \theta)\right )d \xi\right|\\
	&\leq |U_h| \cdot L_{f}' \left|\sum_{\theta\in \Theta}\pi_t(\theta)-\delta_{\theta^c}(\theta)\right|\rightarrow 0,
\end{align*}
w.p.1 ($\mathbb{P}_{\theta^c}^{\infty}$) as $t\rightarrow \infty$, using the consistency of $\pi_t(\theta)$ from \cref{lemma: pi_consistency}.
\end{proof}

% \section*{Acknowledgments}
% The authors gratefully acknowledge the support by the Air Force Office of Scientific Research under Grant FA9550-19-1-0283, Grant FA9550-22-1-0244, and National Science Foundation under Grant DMS2053489.

\bibliographystyle{siamplain}
\bibliography{references}
\end{document}